\documentclass{amsart}
\usepackage{eurosym}
\usepackage{amsfonts}
\usepackage{amsmath}

\newtheorem{theorem}{Theorem}[section]

\newtheorem{claim}[theorem]{Claim}

\newtheorem{corollary}[theorem]{Corollary}

\newtheorem{proposition}[theorem]{Proposition}

\numberwithin{equation}{section}
 
\begin{document}
\title[Structure of shrinking solitons]{Structure at infinity for shrinking
Ricci solitons}
\author{Ovidiu Munteanu}
\email{ovidiu.munteanu@uconn.edu}
\address{Department of Mathematics, University of Connecticut, Storrs, CT
06268, USA}
\author{Jiaping Wang}
\email{jiaping@math.umn.edu}
\address{School of Mathematics, University of Minnesota, Minneapolis, MN
55455, USA}
\thanks{The first author was partially supported by NSF grant DMS-1506220.
The second author was partially supported by NSF grant DMS-1606820 }

\begin{abstract}
This paper concerns the structure at infinity for complete gradient
shrinking Ricci solitons. It is shown that for such a soliton with bounded
curvature, if the round cylinder $\mathbb{R}\times \mathbb{S}^{n-1}/\Gamma$
occurs as a limit for a sequence of points going to infinity along an end,
then the end is asymptotic to the same round cylinder at infinity. This
result is applied to obtain structural results at infinity for four
dimensional gradient shrinking Ricci solitons. It was previously known that
such solitons with scalar curvature approaching zero at infinity must be
smoothly asymptotic to a cone. For the case that the scalar curvature is
bounded from below by a positive constant, we conclude that along each end
the soliton is asymptotic to a quotient of $\mathbb{R}\times \mathbb{S}^{3}$
or converges to a quotient of $\mathbb{R}^{2}\times \mathbb{S}^{2}$ along
each integral curve of the gradient vector field of the potential function.
\end{abstract}

\maketitle

 

\section{Introduction}

The goal of this paper is to continue our study of complete four dimensional
gradient shrinking Ricci solitons initiated in \cite{MW} and to obtain
further information concerning the structure at infinity of such manifolds.
Recall that a Riemannian manifold $\left( M,g\right) $ is a gradient
shrinking Ricci soliton if there exists a smooth function $f\in C^{\infty
}\left( M\right) $ such that the Ricci curvature $\mathrm{Ric}$ of $M$ and
the hessian $\mathrm{Hess}(f)$ of $f$ satisfy the following equation

\begin{equation*}
\mathrm{Ric}+\mathrm{Hess}\left( f\right) =\frac{1}{2}g.
\end{equation*}

By defining $\phi _{t}$ to be the one-parameter family of diffeomorphisms
generated by the vector field $\frac{\nabla f}{-t}$ for $-\infty <t<0,$ one
checks that $g(t)=\left( -t\right) \,\phi _{t}^{\ast }\,g$ is a solution to
the Ricci flow 
\begin{equation*}
\frac{\partial g(t)}{\partial t}=-2\,\mathrm{Ric}(t)
\end{equation*}%
on time interval $(-\infty ,0).$ Since the Ricci flow equation is invariant
under the action of the diffeomorphism group, such solution $g(t)$ is
evidently a shrinking self-similar solution to the Ricci flow. Gradient
shrinking Ricci solitons have played a crucial role in the singularity
analysis of Ricci flows. A conjecture, generally attributed to Hamilton,
asserts that the blow-ups around a type-I singularity point of a Ricci flow
always converge to (nontrivial) gradient shrinking Ricci solitons. More
precisely, a Ricci flow solution $(M,g(t))$ on a finite-time interval $%
[0,T), $ $T<\infty ,$ is said to develop a Type-I singularity (and $T$ is
called a Type-I singular time) if there exists a constant $C>0$ such that
for all $t\in \lbrack 0,T)$ 
\begin{equation*}
\sup_{M} \vert\mathrm{Rm}_{g(t)}\vert_{g(t)}\leq \frac{C}{T-t}
\end{equation*}
and 
\begin{equation*}
\limsup_{t\to T} \,\,\sup_{M} \vert\mathrm{Rm}_{g(t)}\vert_{g(t)} = \infty.
\end{equation*}
Here $\mathrm{Rm}_{g(t)}$ denotes the Riemannian curvature tensor of the
metric $g(t).$ A point $p\in M$ is a singular point if there exists no
neighborhood of $p$ on which $\vert\mathrm{Rm}_{g(t)}\vert_{g(t)}$ stays
bounded as $t\to T.$ Then the conjecture claims that for every sequence $%
\lambda_j\to \infty,$ the rescaled Ricci flows $(M, g_j(t), p)$ defined on $%
[-\lambda_j\,T, 0)$ by $g_j(t) := \lambda_j\,g(T + \lambda_j^{-1}\,t)$
subconverge to a nontrivial gradient shrinking Ricci soliton.

While the conjecture was first confirmed by Perelman \cite{P1} for the
dimension three case, in the most general form it has also been satisfactorily
resolved. In the case where the blow-up limit is compact, it was confirmed
by Sesum \cite{S}. In the general case, blow-up to a gradient shrinking
soliton was proved by Naber \cite{Na}. The nontriviality issue of the
soliton was later taken up by Enders, M\"{u}ller and Topping \cite{EMT}, see
also Cao and Zhang \cite{CZh}.

In view of their importance, it is then natural to seek a classification of
the gradient shrinking Ricci solitons. It is relatively simple to classify
two dimensional ones, \cite{H1}.

\begin{theorem}
A two dimensional gradient shrinking Ricci soliton is isometric to the plane 
$\mathbb{R}^{2}$ or to a quotient of the sphere $\mathbb{S}^{2}.$
\end{theorem}

For the three dimensional case, there is a parallel classification result as
well.

\begin{theorem}
\label{dim3}A three dimensional gradient shrinking Ricci soliton is
isometric to the Euclidean space $\mathbb{R}^{3}$or to a quotient of the
sphere $\mathbb{S}^{3}$or of the cylinder $\mathbb{R}\times \mathbb{S}^{2}.$
\end{theorem}

This theorem has a long history. Ivey \cite{I1} first showed that a three
dimensional compact gradient shrinking Ricci soliton must be a quotient of
the sphere $\mathbb{S}^{3}.$ Later, it was realized from the Hamilton-Ivey
estimate \cite{H1} that the curvature of a three dimensional gradient
shrinking Ricci soliton must be nonnegative. Moreover, by the strong maximum
principle of Hamilton \cite{H3}, the manifold must split off a line, hence
is a quotient of $\mathbb{R}\times \mathbb{S}^{2}$ or $\mathbb{R}^{3},$ if
its sectional curvature is not strictly positive. When the sectional
curvature is strictly positive, Perelman \cite{P2} showed that the soliton
must be compact, hence a quotient of the sphere, provided that the soliton
is noncollapsing with bounded curvature. Obviously, the classification
result follows by combining all these together, at least for the ones which
are noncollapsing with bounded curvature. The result in particular implies
that a type I singularity of the Ricci flow on a compact three dimensional
manifold is necessarily of spherical or neck-like, a fact crucial for
Perelman \cite{P2} to define the Ricci flows with surgery and for the
eventual resolution of the Poincar\'{e} or the more general Thurston's
geometrization conjecture. The noncollapsing assumption was later removed by
Naber \cite{Na}. By adopting a different argument, Ni and Wallach \cite{NW},
and Cao, Chen and Zhu \cite{CCZ} showed the full classification result
Theorem \ref{dim3}. Some relevant contributions were also made in \cite{Na,
PW}. In passing, we mention that it is now known that a complete shrinking
Ricci soliton of any dimension with positive sectional curvature is compact
by \cite{MW2}.

The logical next step is to search for a classification of four dimensional
gradient shrinking Ricci solitons. Such a result should be very much
relevant in understanding the formation of singularities of the Ricci flows
on four dimensional manifolds, just like the three dimensional case.
However, in contrast to the dimension three case, for dimension four or
higher, the curvature of a gradient shrinking Ricci soliton may change sign
as demonstrated by the examples constructed in \cite{FIK}. The existence of
such examples, which are obviously not of the form of a sphere, or the
Euclidean space, or their product, certainly complicates the classification
outlook.

Note that in the case of dimension three, the curvature operator, being
nonnegative, is bounded by the scalar curvature. In the case of dimension
four, we showed that such a conclusion still holds even though the curvature
operator no longer has a fixed sign, \cite{MW}. In particular, this implies
that the curvature operator must be bounded if the scalar curvature is.

\begin{theorem}
\label{MW}Let $\left( M,g,f\right) $ be a four dimensional complete gradient
shrinking Ricci soliton with bounded scalar curvature $S$. Then there exists
a constant $c>0$ so that 
\begin{equation*}
\left\vert \mathrm{Rm}\right\vert \leq c\,S\text{ \ on } M.
\end{equation*}
\end{theorem}

In the theorem, the constant $c>0$ depends only on the upper bound of the
scalar curvature $A$ and the geometry of the geodesic ball $B_{p}(r_{0}),$
where $p$ is a minimum point of potential function $f$ and $r_{0}$ is
determined by $A.$ We stress that the potential function $f$ of the soliton
is exploited in an essential way in our proof by working on the level sets
of $f.$

As an application, we obtained the following structural result. Recall that
a Riemannian cone is a manifold $[0,\infty )\times \Sigma $ endowed with
Riemannian metric $g_{c}=dr^{2}+r^{2}\,g_{\Sigma },$ where $(\Sigma
,g_{\Sigma })$ is a closed $(n-1)$-dimensional Riemannian manifold. Denote $%
E_{R}=(R,\infty )\times \Sigma $ for $R\geq 0$ and define the dilation by $%
\lambda $ to be the map $\rho _{\lambda }:E_{0}\rightarrow E_{0}$ given by $%
\rho _{\lambda }(r,\sigma )=(\lambda \,r,\sigma ).$ Then Riemannian manifold 
$(M,g)$ is said to be $C^{k}$ asymptotic to the cone $(E_{0},g_{c})$ if, for
some $R>0,$ there is a diffeomorphism $\Phi :E_{R}\rightarrow M\setminus
\Omega $ such that $\lambda ^{-2}\,\rho _{\lambda }^{\ast }\,\Phi ^{\ast
}\,g\rightarrow g_{c}$ as $\lambda \rightarrow \infty $ in $%
C_{loc}^{k}(E_{0},g_{c}),$ where $\Omega $ is a compact subset of $M.$ The
following result was established in \cite{MW}.

\begin{theorem}
\label{cone_i} Let $\left( M,g,f\right) $ be a complete four dimensional
gradient shrinking Ricci soliton with scalar curvature converging to zero at
infinity. Then there exists a cone $E_{0}$ such that $(M,g)$ is $C^{k}$
asymptotic to $E_{0}$ for all $k.$
\end{theorem}

A recent result due to Kotschwar and L. Wang \cite{KW} states that two
gradient shrinking Ricci solitons (of arbitrary dimensions) must be
isometric if they are $C^2$ asymptotic to the same cone. Together with our
result, this implies that the classification problem for four dimensional
gradient shrinking Ricci solitons with scalar curvature going to zero at
infinity is reduced to the one for the limiting cones.

In this paper, we take up the case that the scalar curvature is bounded from
below by a positive constant and show the following structural result. Here,
a Riemannian manifold $(M,g)$ is said to be $C^{k}$ asymptotic to the
cylinder $L=\left(\mathbb{R}\times N, g_c\right)$, where $g_c$ is the
product metric, if there is a diffeomorphism $\Phi: L_0=(0,\infty)\times N
\rightarrow M\setminus \Omega$ such that $\rho _{\lambda }^{\ast }\,\Phi
^{\ast}\,g\rightarrow g_{c}$ as $\lambda \rightarrow \infty$ in $%
C_{loc}^{k}(L_{0},g_{c}),$ where $\Omega $ is a compact subset of $M$ and $%
\rho _{\lambda }: L\rightarrow L$ is the translation given by $\rho
_{\lambda }(r,\sigma )=(\lambda+r,\sigma)$ for $r\in \mathbb{R}$ and $\sigma
\in N.$

\begin{theorem}
\label{dim4}Let $\left( M,g,f\right) $ be a complete, four dimensional
gradient shrinking Ricci soliton with bounded scalar curvature $S.$ If $S$
is bounded from below by a positive constant on end $E$ of $M,$ then $E$ is
smoothly asymptotic to the round cylinder $\mathbb{R}\times \mathbb{S}%
^{3}/\Gamma,$ or for any sequence $x_{i}\in E$ going to infinity along an
integral curve of $\nabla f,$ $(M,g,x_{i})$ converges smoothly to $\mathbb{R}%
^{2}\times \mathbb{S}^{2}$ or its $\mathbb{Z}_2$ quotient. Moreover, the
limit is uniquely determined by the integral curve and is independent of the
sequence $x_i.$
\end{theorem}

Here and throughout the paper, $\mathbb{S}^n$ denotes the $n$-dimensional
standard sphere with metric normalized such that $\mathrm{Ric}=\frac{1}{2}%
\,g.$ As pointed out in \cite{Na}, in the second case, the limit in general
may depend on the integral curve as demonstrated by the example $M=\mathbb{R}%
\times \left(\mathbb{R}\times \mathbb{S}^{2}\right)/\mathbb{Z}_2.$ We remark
that under the additional assumption that the Ricci curvature is
non-negative, a similar version of Theorem \ref{dim4} was proved in \cite{Na}
by a different argument. Obviously, Theorem \ref{dim4} together with Theorem %
\ref{cone_i} would provide a description of the geometry at infinity for all
four dimensional gradient shrinking Ricci solitons with bounded scalar
curvature if one could establish a dichotomy that the scalar curvature $S$
either goes to $0$ at infinity or is bounded from below by a positive
constant. This question remains open presently. 

Let us now briefly describe how Theorem \ref{dim4} is proven. According to 
\cite{Na}, for any $n$-dimensional shrinking gradient Ricci soliton $%
(M^{n},g,f)$ with bounded curvature and a sequence of points $x_{i}\in M$
going to infinity along an integral curve of $\nabla f,$ by choosing a
subsequence if necessary, $(M^{n},g,x_{i})$ converges smoothly to a product
manifold $\mathbb{R}\times N^{n-1},$ where $N$ is a gradient shrinking Ricci
soliton. By the classification result of three dimensional gradient
shrinking Ricci solitons Theorem \ref{dim3} and the fact that the scalar
curvature is assumed to be bounded from below by a positive constant,
Theorem \ref{dim4} will then follow from the following structural result for
gradient shrinking Ricci solitons of arbitrary dimension.

\begin{theorem}
\label{cyl}Let $\left( M,g,f\right) $ be an $n$-dimensional, complete,
gradient shrinking Ricci soliton with bounded curvature. Assume that along
an end $E$ of $M$ there exists a sequence of points $x_{i}\rightarrow \infty 
$ with $\left( M,g,x_{i}\right) $ converging to the round cylinder $\mathbb{R%
}\times \mathbb{S}^{n-1}/\Gamma .$ Then $E$ is smoothly asymptotic to the
same round cylinder.
\end{theorem}

The proof of this theorem constitutes the major part of the paper.
Conceptually speaking, we will view the level sets of $f$ endowed with the
induced metric as an approximate Ricci flow and adopt the argument due to
Huisken \cite{Hu} who proved that the Ricci flow starting from a manifold
with sufficiently pinched sectional curvature must converge to a quotient of
the round sphere. However, to actually carry out the argument, we have to
overcome some serious technical hurdles, one of them being the control of
the scalar curvature. Along the way, we have managed to obtain some
localized estimates for the derivatives of the curvature tensor, which may
be of independent interest. In particular, these estimates enabled us to
derive a Harnack type estimate for the scalar curvature.

There are quite a few related works concerning the geometry and
classification of high dimensional gradient shrinking Ricci solitons. The
survey paper \cite{C1} contains a wealth of information and then current
results. The paper by Naber \cite{Na} has strong influence on the present
work. For some of the more recent progress, we refer to \cite{Cat1}, \cite%
{Cat2}, \cite{CC1}, \cite{CC2}, \cite{CWZ}, \cite{LNW}. In the other
direction, Catino, Deruelle and Mazzieri \cite{CDM} have attempted to
address the rigidity issue for the complete gradient shrinking Ricci
solitons which are asymptotic to the round cylinder at infinity, that is,
whether the soliton $M$ in Theorem \ref{cyl} is in fact itself a round
cylinder. Apparently, this issue remains unresolved, as it was stated there
that the proof given is yet incomplete.

The paper is organized as follows. After recalling a few preliminary facts
in section \ref{Prelim}, we prove some useful localized curvature estimates
in section \ref{Curvature}. Theorem \ref{cyl} is then proved in section \ref%
{Convergence}. The applications to four dimensional gradient shrinking Ricci
solitons are discussed in section \ref{4-dim}.

We would like to thank Ben Chow, Brett Kotschwar, Aaron Naber and Lei Ni for
their interest and helpful discussions. We also thank Yongjia Zhang for
useful comments on a previous draft of this paper.

\section{\label{Prelim}Preliminaries}

In this section, we recall some preliminary facts concerning gradient
shrinking Ricci solitons. We will use the same notation as in \cite{MW}.
Throughout this paper, $(M,g)$ denotes an $n$-dimensional, complete
noncompact gradient shrinking Ricci soliton. A result of Chen (\cite{C,C2})
implies that the scalar curvature $S>0$ on $M$, unless $M$ is flat. This
result was later refined in \cite{CLY} to that 
\begin{equation}
S\geq \frac{C_{0}}{f}  \label{CLY}
\end{equation}%
for some positive constant $C_{0}$ depending on the soliton. Furthermore, by
adding a constant to the potential function $f$ if necessary, one has the
following important identity due to Hamilton \cite{H1}. 
\begin{equation*}
S+\left\vert \nabla f\right\vert ^{2}=f.
\end{equation*}%
For such a normalized potential function $f$, it is well known \cite{CZ}
that there exist positive constants $c_{1}$ and $c_{2}$ such that 
\begin{equation}
\frac{1}{4}r^{2}\left( x\right) -c_{1}\,r\left( x\right) -c_{2}\leq f\left(
x\right) \leq \frac{1}{4}r^{2}\left( x\right) +c_{1}\,r\left( x\right)
+c_{2},  \label{asympt}
\end{equation}%
where $r\left( x\right) $ is the distance to a fixed point $p\in M.$
Moreover, $c_{1}$ and $c_{2}$ can be chosen to depend only on $n$ if $p$ is
a minimum point of $f$, see \cite{HM}.

Consequently, if the scalar curvature of $\left( M,g\right) $ is bounded,
then there exists $t_{0}>0$ such that the level set 
\begin{equation*}
\Sigma \left( t\right) =\left\{ x\in M:f\left( x\right) =t\right\} 
\end{equation*}%
of $f$ is a compact Riemannian manifold for $t\geq t_{0}.$ Also, the domain 
\begin{equation*}
D\left( t\right) :=\left\{ x\in M:f\left( x\right) \leq t\right\} 
\end{equation*}%
is a compact manifold with smooth boundary $\Sigma \left( t\right) .$

Since the volume of $M$ grows polynomially of order at most $n$ by \cite{CZ}%
, one sees that the weighted volume of $M$ given by 
\begin{equation*}
V_f(M)=\int_M e^{-f}\, dv
\end{equation*}
must be finite.

We recall the following equations for various curvature quantities of $M$,
see e.g. \cite{MW}. 
\begin{eqnarray}
\nabla S &=&2\mathrm{Ric}(\nabla f)  \label{m8} \\
\nabla _{l}R_{ijkl} &=&R_{ijkl}f_{l}=\nabla _{j}R_{ik}-\nabla _{i}R_{jk} 
\notag \\
\Delta _{f}S &=&S-2\left\vert \mathrm{Ric}\right\vert ^{2}  \notag \\
\Delta _{f}\mathrm{Ric} &=&\mathrm{Ric}-2\mathrm{RmRic}  \notag \\
\Delta _{f}\mathrm{Rm} &=&\mathrm{Rm}+\mathrm{Rm}\ast \mathrm{Rm}  \notag \\
\Delta _{f}\left( \nabla ^{k}\mathrm{Rm}\right)  &=&\left( \frac{k}{2}%
+1\right) \nabla ^{k}\mathrm{Rm}+\sum_{j=0}^{k}\nabla ^{j}\mathrm{Rm}\ast
\nabla ^{k-j}\mathrm{Rm}.  \notag
\end{eqnarray}%
Here, $\Delta _{f}$ is the weighted Laplacian defined by $\Delta
_{f}T=\Delta T-\left\langle \nabla f,\nabla T\right\rangle $  for a tensor
field $T.$ The notation $\mathrm{Rm}\ast \mathrm{Rm}$ denotes a quadratic
expression in the Riemann curvature tensor and $\nabla ^{j}\mathrm{Rm}$
denotes the $j$-th covariant derivative of the curvature tensor $\mathrm{Rm}.
$

As mentioned in the introduction, $M$ may be viewed as a self-similar
solution to the Ricci flow. Therefore, if the curvature of $\left(
M,g\right) $ is bounded, that is, there exists a constant $C>0$ such that $%
\left\vert \mathrm{Rm}\right\vert \leq C$ on $M$, then by Shi's derivative
estimates \cite{Shi}, for each $k\geq 1,$ there exists a constant $A_{k}>0$
such that 
\begin{equation}
\left\vert \nabla ^{k}\mathrm{Rm}\right\vert \leq A_{k}\text{ \ on }M
\label{Rm}
\end{equation}
with $A_{k}$ depending only on $n,k$ and $C.$

Using (\ref{m8}) we get, for any $k\geq 0$ and $\sigma >0$, that 
\begin{align}
& \Delta _{f}\left( \left\vert \nabla ^{k}\mathrm{Rm}\right\vert
^{2}S^{-\sigma }\right) \geq S^{-\sigma }\left( 2\left\vert \nabla ^{k+1}%
\mathrm{Rm}\right\vert ^{2}+\left( k+2\right) \left\vert \nabla ^{k}\mathrm{%
Rm}\right\vert ^{2}\right)  \label{iw0} \\
& -cS^{-\sigma }\Sigma _{j=0}^{k}\left\vert \nabla ^{j}\mathrm{Rm}%
\right\vert \left\vert \nabla ^{k-j}\mathrm{Rm}\right\vert \left\vert \nabla
^{k}\mathrm{Rm}\right\vert  \notag \\
& +\left\vert \nabla ^{k}\mathrm{Rm}\right\vert ^{2}\left( -\sigma
S^{-\sigma }+2\sigma \left\vert \mathrm{Ric}\right\vert ^{2}S^{-\sigma
-1}+\sigma \left( \sigma +1\right) \left\vert \nabla S\right\vert
^{2}S^{-\sigma -2}\right)  \notag \\
& +2\left\langle \nabla \left\vert \nabla ^{k}\mathrm{Rm}\right\vert
^{2},\nabla S^{-\sigma }\right\rangle .  \notag
\end{align}%
Observe that%
\begin{equation*}
2\left\langle \nabla \left\vert \nabla ^{k}\mathrm{Rm}\right\vert
^{2},\nabla S^{-\sigma }\right\rangle \geq -2\left\vert \nabla ^{k+1}\mathrm{%
Rm}\right\vert ^{2}S^{-\sigma }-2\sigma ^{2}\left\vert \nabla S\right\vert
^{2}S^{-\sigma -2}\left\vert \nabla ^{k}\mathrm{Rm}\right\vert ^{2}.
\end{equation*}%
This implies the function $w:=\left\vert \nabla ^{k}\mathrm{Rm}\right\vert
^{2}S^{-\sigma }$ satisfies 
\begin{eqnarray}
\Delta _{f}w &\geq &\left( k+2-\sigma +\left( \sigma -\sigma ^{2}\right)
\left\vert \nabla \ln S\right\vert ^{2}\right) w  \label{iw1} \\
&&-c\Sigma _{j=0}^{k}\left\vert \nabla ^{j}\mathrm{Rm}\right\vert \left\vert
\nabla ^{k-j}\mathrm{Rm}\right\vert \left\vert \nabla ^{k}\mathrm{Rm}%
\right\vert S^{-\sigma }.  \notag
\end{eqnarray}

If instead in (\ref{iw0}) we use

\begin{eqnarray*}
&&2\left\langle \nabla \left\vert \nabla ^{k}\mathrm{Rm}\right\vert
^{2},\nabla S^{-\sigma }\right\rangle \\
&=&\left\langle \nabla \left( \left\vert \nabla ^{k}\mathrm{Rm}\right\vert
^{2}S^{-\sigma }S^{\sigma }\right) ,\nabla S^{-\sigma }\right\rangle
+\left\langle \nabla \left\vert \nabla ^{k}\mathrm{Rm}\right\vert
^{2},\nabla S^{-\sigma }\right\rangle \\
&\geq &\left\langle \nabla \left( \left\vert \nabla ^{k}\mathrm{Rm}%
\right\vert ^{2}S^{-\sigma }\right) ,\nabla S^{-\sigma }\right\rangle
S^{\sigma }+\left\vert \nabla ^{k}\mathrm{Rm}\right\vert ^{2}S^{-\sigma
}\left\langle \nabla S^{\sigma },\nabla S^{-\sigma }\right\rangle \\
&&-2\sigma \left\vert \nabla ^{k+1}\mathrm{Rm}\right\vert \left\vert \nabla
S\right\vert \left\vert \nabla ^{k}\mathrm{Rm}\right\vert S^{-\sigma -1} \\
&\geq &-\sigma \left\langle \nabla \left( \left\vert \nabla ^{k}\mathrm{Rm}%
\right\vert ^{2}S^{-\sigma }\right) ,\nabla \ln S\right\rangle -\frac{3}{2}%
\sigma ^{2}\left\vert \nabla ^{k}\mathrm{Rm}\right\vert ^{2}\left\vert
\nabla S\right\vert ^{2}S^{-\sigma -2} \\
&&-2\left\vert \nabla ^{k+1}\mathrm{Rm}\right\vert ^{2}S^{-\sigma },
\end{eqnarray*}%
then the function $w:=\left\vert \nabla ^{k}\mathrm{Rm}\right\vert
^{2}S^{-\sigma }$ satisfies 
\begin{eqnarray}
\Delta _{F}w &\geq &\left( k+2-\sigma +\left( \sigma -\frac{\sigma ^{2}}{2}%
\right) \left\vert \nabla \ln S\right\vert ^{2}\right) w  \label{iw} \\
&&-c\Sigma _{j=0}^{k}\left\vert \nabla ^{j}\mathrm{Rm}\right\vert \left\vert
\nabla ^{k-j}\mathrm{Rm}\right\vert \left\vert \nabla ^{k}\mathrm{Rm}%
\right\vert S^{-\sigma },  \notag
\end{eqnarray}%
where $F:=f-\sigma \ln S.$

\section{\label{Curvature}Curvature estimates for shrinkers}

In this section, we establish some localized derivative estimates for the
curvature tensor of a gradient shrinking Ricci soliton. The estimates will
be applied in next section to prove Theorem \ref{cyl}.

Throughout this section, $(M,g)$ denotes an $n$-dimensional gradient
shrinking Ricci soliton with bounded curvature. Hence, we may assume that (%
\ref{Rm}) holds everywhere on $M.$

Everywhere in this paper, we will denote by $\left\{ e_{1},e_{2},\cdots
,e_{n}\right\} $ a local orthonormal frame of $M$ with 
\begin{equation*}
e_{n}:=\frac{\nabla f}{\left\vert \nabla f\right\vert }.
\end{equation*}%
Clearly, $e_{n}$ is a unit normal vector to $\Sigma \left( t\right) $ and $%
\left\{ e_{1},e_{2},\cdots ,e_{n-1}\right\} $ a local orthonormal frame of $%
\Sigma \left( t\right) .$ Throughout this paper, the indices $%
a,b,c,d=1,2,\cdots ,n-1$ and $i,j,k,l=1,2,\cdots ,n$.
 In this notation, the second fundamental form of $%
\Sigma \left( t\right) $ is given by 
\begin{equation}
h_{ab}=\frac{f_{ab}}{\left\vert \nabla f\right\vert },  \label{2ff}
\end{equation}%
for any $a,b=1,2,\cdots ,n-1.$

By (\ref{m8}) we have that 
\begin{eqnarray}
\left\vert R_{ijkn}\right\vert &=&\frac{\left\vert R_{ijkl}f_{l}\right\vert 
}{\left\vert \nabla f\right\vert }  \label{d1} \\
&=&\frac{1}{\left\vert \nabla f\right\vert }\left\vert \nabla
_{j}R_{ik}-\nabla _{i}R_{jk}\right\vert  \notag \\
&\leq &\frac{2\left\vert \nabla \mathrm{Ric}\right\vert }{\left\vert \nabla
f\right\vert }.  \notag
\end{eqnarray}

Denote with

\begin{eqnarray}
&&\overset{\circ }{R}_{ab}:=R_{ab}-\frac{1}{n-1}S\,g_{ab},  \label{U,V,W} \\
&&U_{abcd}:=\frac{1}{\left( n-1\right) \left( n-2\right) }S\left(
g_{ac}g_{bd}-g_{ad}g_{bc}\right),  \notag \\
&&\overset{\circ }{R}_{abcd}:=R_{abcd}-U_{abcd},  \notag \\
&&V_{abcd}:=\frac{1}{n-3}\left( \overset{\circ }{R}_{ac}g_{bd}+\overset{%
\circ }{R}_{bd}g_{ac}-\overset{\circ }{R}_{ad}g_{bc}-\overset{\circ }{R}%
_{bc}g_{ad}\right),  \notag \\
&&W_{abcd}:=R_{abcd}-U_{abcd}-V_{abcd},  \notag
\end{eqnarray}%
where $a,b,c,d=1,2,\cdots ,n-1$.
It should be pointed out that $W$ is not the Weyl curvature tensor of the
manifold $\left( M,g\right)$, restricted to the level set $\Sigma \left( t\right)$, rather it is an approximation of the one of $\Sigma \left( t\right)$.

Denote 
\begin{eqnarray*}
&&\left\vert \overset{\circ }{\mathrm{Ric}_{\Sigma }}\right\vert
^{2}:=\left\vert \overset{\circ }{R}_{ab}\right\vert ^{2}, \\
&&\left\vert \overset{\circ }{\mathrm{Rm}}_{\Sigma }\right\vert
^{2}:=\left\vert \overset{\circ }{R}_{abcd}\right\vert ^{2}.
\end{eqnarray*}%
We now state the main result of this section. Fix $t_{0}>0$ large enough,
depending only on dimension $n$ and the constant $A_{0}$ in (\ref{Rm}).
Since $S\leq nA_{0}$, using Hamilton's identity $S+\left\vert \nabla
f\right\vert ^{2}=f$ we get that the level sets $\Sigma \left( t\right) $ of 
$\ f$ are all smooth for $t\geq t_{0}$. Also fix some $T>t_{0}$. We have the
following.

\begin{theorem}
\label{Curv}Let $\left( M,g,f\right) $ be an $n$-dimensional, complete,
gradient shrinking Ricci soliton with bounded curvature such that 
\begin{eqnarray}
\ \left\vert \overset{\circ }{\mathrm{Rm}}_{\Sigma }\right\vert ^{2} &\leq
&\eta _{1}S^{2}\text{ \ on }D\left( T\right) \backslash D\left( t_{0}\right),
\label{p} \\
S &\geq &\eta _{2}\text{ \ on }\Sigma \left( t_{0}\right)  \notag
\end{eqnarray}%
for some $\eta _{1},\eta _{2}>0.$ Then for each $k\geq 0,$ there exists
constant $c_{k}>0$ such that 
\begin{equation*}
\left\vert \nabla ^{k}\mathrm{Rm}\right\vert ^{2}\leq c_{k}\,S^{k+2}\text{ \
on }D\left( T\right) \backslash D\left( t_{0}\right).
\end{equation*}
\end{theorem}

For given $C_{0}$ from (\ref{CLY}) and $A_{k}$ from (\ref{Rm}), the constant 
$c_{k}$ in Theorem \ref{Curv} only depends on $n$, $\eta _{1}$, $\eta _{2}$, 
$C_{0}$ and $A_{0}$,....,$A_{Kk},$ where $K$ is an absolute constant ($K=100$
suffices). We stress that all $c_k$ are independent of $t_0$ and $T$. 

As a useful
corollary of this theorem, if

\begin{eqnarray}
\ \left\vert \mathrm{Rm}\right\vert ^{2} &\leq &\eta _{1}S^{2}\text{ \ \ on }%
D\left( T\right) \backslash D\left( t_{0}\right),  \label{p'} \\
S &\geq &\eta _{2}\text{ \ on }\Sigma \left( t_{0}\right),  \notag
\end{eqnarray}%
then 
\begin{equation*}
\left\vert \nabla ^{k}\mathrm{Rm}\right\vert ^{2}\leq c_{k}S^{k+2}\text{ \
on }D\left( T\right) \backslash D\left( t_{0}\right).
\end{equation*}%
In fact, in the proof of Theorem \ref{Curv}, we will show that (\ref{p})
implies (\ref{p'}). The converse is obviously true.

According to Theorem \ref{MW}, (\ref{p'}) is true for $T=\infty $ on a four
dimensional gradient shrinking Ricci soliton with bounded scalar curvature.
Hence, we obtain the following.

\begin{corollary}
\label{curv_4d}Let $\left( M,g,f\right) $ be a complete, four dimensional,
gradient shrinking Ricci soliton with bounded scalar curvature. Then for
each $k\geq 0,$ there exists constant $c_{k}>0$ so that 
\begin{equation*}
\left\vert \nabla ^{k}\mathrm{Rm}\right\vert ^{2}\leq c_{k}\,S^{k+2}\text{ \
on }M.
\end{equation*}
\end{corollary}

The rest of the section is devoted to proving Theorem \ref{Curv}. First, we
observe that $T$ may be assumed to be large compared to $t_{0}.$ Indeed,
define $\phi _{t}$ by 
\begin{eqnarray*}
\frac{d\phi _{t}}{dt} &=&\frac{\nabla f}{\left\vert \nabla f\right\vert ^{2}}
\\
\phi _{t_{0}} &=&\mathrm{Id}\text{ \ on }\Sigma \left( t_{0}\right).
\end{eqnarray*}%
For a fixed $x\in \Sigma \left( t_{0}\right),$ denote $S\left( t\right)
:=S\left( \phi _{t}\left( x\right) \right),$ where $t\geq t_{0}.$ Then, as $%
\left\langle \nabla S,\nabla f\right\rangle =\Delta S-S+2\left\vert \mathrm{%
Ric}\right\vert ^{2},$ it follows from (\ref{Rm}) that

\begin{equation}
\left\vert \frac{dS}{dt}\right\vert =\frac{\left\vert \left\langle \nabla
S,\nabla f\right\rangle \right\vert }{\left\vert \nabla f\right\vert ^{2}}%
\leq \frac{c_{0}}{t}.  \label{y0}
\end{equation}%
Integrating this in $t$ we get 
\begin{eqnarray}
S\left( t\right) &\geq &S\left( t_{0}\right) -c_{0}\ln \frac{t}{t_{0}}
\label{y1} \\
&\geq &\eta _{2}-c_{0}\ln \frac{t}{t_{0}}.  \notag
\end{eqnarray}%
Hence, if $T\leq e^{\frac{\eta _{2}}{2c_{0}}}t_{0},$ then (\ref{y1}) implies 
$S\geq \frac{1}{2}\eta _{2}$ on $D\left( T\right) \backslash D\left(
t_{0}\right) .$ In this case, Theorem \ref{Curv} follows directly from (\ref%
{Rm}). So we may assume from now on that there exists $\nu >1$, depending
only on $\eta _{2}$, $A_{0}$ and $A_{2}$, satisfying%
\begin{equation}
T\geq \nu t_{0}.  \label{T}
\end{equation}%
Using (\ref{d1}) and (\ref{CLY}), we see from (\ref{p}) that 
\begin{equation}
\left\vert \mathrm{Rm}\right\vert \leq c\left( S+\frac{\left\vert \nabla 
\mathrm{Ric}\right\vert }{\sqrt{f}}\right) \leq c\sqrt{S}\ \ \text{on }%
D\left( T\right) \backslash D\left( t_{0}\right) .  \label{R}
\end{equation}%
The proof of Theorem \ref{Curv} is divided in two parts.

\begin{proposition}
\label{D1}Let $\left( M,g,f\right) $ be an $n$-dimensional, complete,
gradient shrinking Ricci soliton such that (\ref{p}) holds. Then, for any $%
k\geq 0,$ there exists constant $c_{k}$ such that 
\begin{equation}
\left\vert \nabla ^{k}\mathrm{Rm}\right\vert \leq c_{k}S\text{ \ on }D\left(
T\right) \backslash D\left( t_{0}\right) .  \label{I}
\end{equation}
\end{proposition}

\begin{proof}
We first prove by induction on $k\geq 0$ that 
\begin{equation}
\left\vert \nabla ^{k}\mathrm{Rm}\right\vert ^{2}\leq c_{k}S\text{ \ on }%
D\left( a_{k}T\right) \backslash D\left( t_{0}\right) ,  \label{y2}
\end{equation}%
where 
\begin{equation*}
a_{k}:=1-\left( 1-\frac{1}{2^{k}}\right) \frac{1}{\sqrt{T}}.
\end{equation*}%
For $k=0$ we get (\ref{y2}) from (\ref{R}). Let us assume (\ref{y2}) is true
for $k=0,1,..,l-1$ and prove it for $k=l.$ By (\ref{iw1}), on $D\left(
T\right) \backslash D\left( t_{0}\right) $ the function $w:=\left\vert
\nabla ^{l}\mathrm{Rm}\right\vert ^{2}S^{-1}$ satisfies%
\begin{eqnarray*}
\Delta _{f}w &\geq &2w-c\Sigma _{j=0}^{l}\left\vert \nabla ^{j}\mathrm{Rm}%
\right\vert \left\vert \nabla ^{l-j}\mathrm{Rm}\right\vert \left\vert \nabla
^{l}\mathrm{Rm}\right\vert S^{-1} \\
&\geq &w-c\Sigma _{j=0}^{l}\left\vert \nabla ^{j}\mathrm{Rm}\right\vert
^{2}\left\vert \nabla ^{l-j}\mathrm{Rm}\right\vert ^{2}S^{-1}.
\end{eqnarray*}%
By the induction hypothesis,%
\begin{equation*}
\left\vert \nabla ^{j}\mathrm{Rm}\right\vert ^{2}\left\vert \nabla ^{l-j}%
\mathrm{Rm}\right\vert ^{2}S^{-1}\leq c\text{ \ on }D\left( a_{l-1}T\right)
\backslash D\left( t_{0}\right) .
\end{equation*}%
Hence, it follows from above that%
\begin{equation}
\Delta _{f}w\geq w-c_{l}\text{ \ \ on }D\left( a_{l-1}T\right) \backslash
D\left( t_{0}\right) .  \label{y5}
\end{equation}%
Define the cut-off function 
\begin{equation}
\psi \left( f\left( x\right) \right) =\frac{e^{\sqrt{a_{l-1}T}}-e^{\frac{f}{%
\sqrt{a_{l-1}T}}}}{e^{\sqrt{a_{l-1}T}}}  \label{psi}
\end{equation}%
with support in $D\left( a_{l-1}T\right) \backslash D\left( t_{0}\right) $
and let $G:=\psi ^{2}w.$ By (\ref{y5}) we get%
\begin{eqnarray}
\Delta _{f}G &\geq &G-c_{l}+2\psi ^{-1}\left( \Delta _{f}\psi \right)
G-6\psi ^{-2}\left\vert \nabla \psi \right\vert ^{2}G  \label{y6} \\
&&+2\psi ^{-2}\left\langle \nabla G,\nabla \psi ^{2}\right\rangle .  \notag
\end{eqnarray}%
Let $x_{0}$ be the maximum point of $G$ on $D\left( a_{l-1}T\right)
\backslash D\left( t_{0}\right) .$ If $x_{0}\in \Sigma \left( t_{0}\right) ,$
then $G\left( x_{0}\right) \leq c_{l}$ by (\ref{p}) and (\ref{Rm}). So,
without loss of generality, we may assume that $x_{0}$ is an interior point.
If at $x_{0}$ we have $\psi ^{-1}\left( \Delta _{f}\psi \right) -3\psi
^{-2}\left\vert \nabla \psi \right\vert ^{2}\geq 0,$ then applying the
maximum principle to (\ref{y6}) we get $G\left( x_{0}\right) \leq c_{l}.$
Now suppose that%
\begin{equation}
\psi ^{-1}\left( \Delta _{f}\psi \right) -3\psi ^{-2}\left\vert \nabla \psi
\right\vert ^{2}<0\text{ \ at }x_{0}.  \label{y6'}
\end{equation}%
Since%
\begin{eqnarray}
\Delta _{f}\psi &=&\psi ^{\prime }\Delta _{f}\left( f\right) +\psi ^{\prime
\prime }\left\vert \nabla f\right\vert ^{2}  \label{y6''} \\
&=&\frac{e^{\frac{f}{\sqrt{a_{l-1}T}}}}{\sqrt{a_{l-1}T}e^{\sqrt{a_{l-1}T}}}%
\left( f-\frac{n}{2}-\frac{\left\vert \nabla f\right\vert ^{2}}{\sqrt{%
a_{l-1}T}}\right)  \notag \\
&\geq &\frac{1}{2}\frac{e^{\frac{f}{\sqrt{a_{l-1}T}}}}{e^{\sqrt{a_{l-1}T}}}%
\frac{f}{\sqrt{a_{l-1}T}},  \notag
\end{eqnarray}%
by (\ref{y6'}) it follows that 
\begin{eqnarray*}
\frac{1}{2}\frac{e^{\frac{f}{\sqrt{a_{l-1}T}}}}{e^{\sqrt{a_{l-1}T}}}\frac{f}{%
\sqrt{a_{l-1}T}}\psi \left( x_{0}\right) &\leq &\psi \Delta _{f}\psi \\
&<&3\left\vert \nabla \psi \right\vert ^{2} \\
&\leq &3\frac{e^{\frac{2f}{\sqrt{a_{l-1}T}}}}{e^{2\sqrt{a_{l-1}T}}}\frac{f}{%
a_{l-1}T}.
\end{eqnarray*}%
This immediately implies 
\begin{equation*}
\psi \left( x_{0}\right) \leq \frac{c}{\sqrt{T}}.
\end{equation*}%
Hence, we obtain from (\ref{CLY}) that%
\begin{equation*}
G\left( x_{0}\right) \leq \frac{c}{T}\left\vert \nabla ^{l}\mathrm{Rm}%
\right\vert ^{2}S^{-1}\leq c_{l}.
\end{equation*}%
In conclusion, this proves 
\begin{equation}
G\leq c_{l}\text{ \ on \ }D\left( a_{l-1}T\right) \backslash D\left(
t_{0}\right) .  \label{y7}
\end{equation}%
Since $\left( a_{l-1}-a_{l}\right) \sqrt{T}=\frac{1}{2^{l}},$ on $D\left(
a_{l}T\right) \backslash D\left( t_{0}\right) $ by (\ref{psi}) we have%
\begin{eqnarray*}
\psi &\geq &1-e^{\frac{1}{\sqrt{a_{l-1}}}\left( a_{l}-a_{l-1}\right) \sqrt{T}%
} \\
&\geq &1-e^{-2^{-l}}.
\end{eqnarray*}%
By (\ref{y7}) we get that $\left\vert \nabla ^{l}\mathrm{Rm}\right\vert
^{2}S^{-1}\leq c_{l}$ on $D\left( a_{l}T\right) \backslash D\left(
t_{0}\right) .$ This completes the induction step and proves (\ref{y2}). In
particular, we have%
\begin{equation}
\left\vert \nabla ^{k}\mathrm{Rm}\right\vert ^{2}\leq c_{k}S\text{ \ on }%
D\left( T-\sqrt{T}\right) \backslash D\left( t_{0}\right) .  \label{y8}
\end{equation}%
We now prove that for all $k\geq 0$ 
\begin{equation}
\left\vert \nabla ^{k}\mathrm{Rm}\right\vert ^{2}\leq c_{k}S^{2}\text{ \ on }%
D\left( T-2\sqrt{T}\right) \backslash D\left( t_{0}\right) .  \label{y9}
\end{equation}%
For $k=0,$ (\ref{y9}) follows from (\ref{R}) and (\ref{y8}). Indeed, by (\ref%
{y8}), one has 
\begin{equation*}
\left\vert \nabla \mathrm{Ric}\right\vert \leq c\,\sqrt{S}
\end{equation*}%
on $D\left( b_{0}T\right) =D\left( T-\sqrt{T}\right) .$ Plugging this into (%
\ref{R}) and using (\ref{CLY}), one sees that 
\begin{equation*}
\left\vert \mathrm{Rm}\right\vert \leq c\,S
\end{equation*}%
on $D\left( b_{0}T\right) .$\ We now prove (\ref{y9}) for $k\geq 1$.

By (\ref{iw}), on $D\left( T-\sqrt{T}\right) \backslash D\left( t_{0}\right)
,$ the function $w:=\left\vert \nabla ^{k}\mathrm{Rm}\right\vert ^{2}S^{-2}$
satisfies%
\begin{eqnarray}
\Delta _{F}w &\geq &w-c\Sigma _{j=0}^{k}\left\vert \nabla ^{j}\mathrm{Rm}%
\right\vert \left\vert \nabla ^{k-j}\mathrm{Rm}\right\vert \left\vert \nabla
^{k}\mathrm{Rm}\right\vert S^{-2}  \label{y10} \\
&\geq &\frac{1}{2}w-c\Sigma _{j=0}^{k}\left\vert \nabla ^{j}\mathrm{Rm}%
\right\vert ^{2}\left\vert \nabla ^{k-j}\mathrm{Rm}\right\vert ^{2}S^{-2} 
\notag \\
&\geq &\frac{1}{2}w-c_{k},  \notag
\end{eqnarray}%
where $F:=f-2\ln S$ and in the last line we have used (\ref{y8}). Let $%
G:=\psi ^{2}w$ with 
\begin{equation}
\psi :=\frac{e^{\sqrt{T-\sqrt{T}}}-e^{\frac{f}{\sqrt{T-\sqrt{T}}}}}{e^{\sqrt{%
T-\sqrt{T}}}}.  \label{psi_1}
\end{equation}%
By (\ref{y10}) we obtain%
\begin{eqnarray}
\Delta _{F}G &\geq &\frac{1}{2}G-c_{k}+2\psi ^{-1}\left( \Delta _{F}\psi
\right) G-6\psi ^{-2}\left\vert \nabla \psi \right\vert ^{2}G  \label{y11} \\
&&+2\psi ^{-2}\left\langle \nabla G,\nabla \psi ^{2}\right\rangle .  \notag
\end{eqnarray}%
Let $x_{0}$ be the maximum point of $G$ on $D\left( T-\sqrt{T}\right)
\backslash D\left( t_{0}\right) .$ As above, we may assume that $x_{0}$ is
an interior point. If at $x_{0}$ we have $\psi ^{-1}\left( \Delta _{F}\psi
\right) -3\psi ^{-2}\left\vert \nabla \psi \right\vert ^{2}\geq 0,$ then the
maximum principle implies $G\left( x_{0}\right) \leq c_{k}.$ So it remains
to consider the case that%
\begin{equation}
\psi ^{-1}\left( \Delta _{F}\psi \right) -3\psi ^{-2}\left\vert \nabla \psi
\right\vert ^{2}<0\text{ \ at }x_{0}.  \label{y12}
\end{equation}%
As in (\ref{y6''}), we have%
\begin{equation}
\Delta _{f}\psi \geq \frac{1}{2}\frac{e^{\frac{f}{\sqrt{T-\sqrt{T}}}}}{e^{%
\sqrt{T-\sqrt{T}}}}\frac{f}{\sqrt{T-\sqrt{T}}}.  \label{y13}
\end{equation}%
Furthermore, 
\begin{equation}
\left\vert \left\langle \nabla \ln S,\nabla \psi \right\rangle \right\vert =%
\frac{e^{\frac{f}{\sqrt{T-\sqrt{T}}}}}{\sqrt{T-\sqrt{T}}e^{\sqrt{T-\sqrt{T}}}%
}\left\vert \left\langle \nabla S,\nabla f\right\rangle \right\vert S^{-1}.
\label{y14}
\end{equation}%
However, (\ref{y8}) implies that 
\begin{eqnarray}
\left\vert \left\langle \nabla S,\nabla f\right\rangle \right\vert  &\leq
&\left\vert \Delta S\right\vert +S+2\left\vert \mathrm{Ric}\right\vert ^{2}
\label{y14'} \\
&\leq &c\sqrt{S}.  \notag
\end{eqnarray}%
For $c>0$ specified in (\ref{y14'}) and $C_{0}$ the constant in (\ref{CLY}),
we consider the following cases. 

First, assume that 
\begin{equation*}
c\sqrt{S}<C_{0}^{\frac{3}{4}}S^{\frac{1}{4}}\text{ \ at }x_{0}.
\end{equation*}%
Then (\ref{y14'}) implies that%
\begin{eqnarray*}
\left\vert \left\langle \nabla S,\nabla f\right\rangle \right\vert
S^{-1}\left( x_{0}\right)  &\leq &\left( C_{0}S^{-1}\left( x_{0}\right)
\right) ^{\frac{3}{4}} \\
&\leq &f^{\frac{3}{4}}\left( x_{0}\right) ,
\end{eqnarray*}%
where the last line follows from (\ref{CLY}). By (\ref{y14}), this implies
that%
\begin{equation*}
\left\vert \left\langle \nabla \ln S,\nabla \psi \right\rangle \right\vert
\left( x_{0}\right) \leq \frac{e^{\frac{f\left( x_{0}\right) }{\sqrt{T-\sqrt{%
T}}}}}{\sqrt{T-\sqrt{T}}e^{\sqrt{T-\sqrt{T}}}}f^{\frac{3}{4}}\left(
x_{0}\right) .
\end{equation*}%
From (\ref{y13}), we get that at $x_{0}$, the maximum point of $G$ on $%
D\left( T-\sqrt{T}\right) \backslash D\left( t_{0}\right) $, we have 
\begin{equation*}
\Delta _{F}\psi \geq \frac{1}{3}\frac{e^{\frac{f}{\sqrt{T-\sqrt{T}}}}}{e^{%
\sqrt{T-\sqrt{T}}}}\frac{f}{\sqrt{T-\sqrt{T}}}.
\end{equation*}%
Plugging this into (\ref{y12}), we get 
\begin{equation*}
\psi \left( x_{0}\right) \leq \frac{c}{\sqrt{T}}.
\end{equation*}%
Therefore, by (\ref{CLY}) and (\ref{y8}),%
\begin{eqnarray*}
G\left( x_{0}\right)  &=&\psi ^{2}\left( x_{0}\right) w\left( x_{0}\right) 
\\
&\leq &\frac{c}{T}\left\vert \nabla ^{k}\mathrm{Rm}\right\vert ^{2}S^{-2} \\
&\leq &c_{k}.
\end{eqnarray*}%
Finally, assume that 
\begin{equation*}
c\sqrt{S}\geq C_{0}^{\frac{3}{4}}S^{\frac{1}{4}}\ \ \text{at }x_{0},
\end{equation*}%
where $c$ is the constant in (\ref{y14'}) and $C_{0}$ the constant in (\ref%
{CLY}). 

We then get $S\left( x_{0}\right) \geq C_{0}^{3}c^{-4}$, which by (%
\ref{Rm}) proves that $G\left( x_{0}\right) \leq c_{k}$. In conclusion, from
these two cases we get that%
\begin{equation}
G\leq c_{k}\text{ \ on \ }D\left( T-\sqrt{T}\right) \backslash D\left(
t_{0}\right) .  \label{y15}
\end{equation}%
Note by (\ref{psi_1}) we have on $D\left( T-2\sqrt{T}\right) \backslash
D\left( t_{0}\right) $ 
\begin{eqnarray*}
\psi  &\geq &1-e^{-\frac{\sqrt{T}}{\sqrt{T-\sqrt{T}}}} \\
&\geq &1-e^{-1}.
\end{eqnarray*}%
So (\ref{y15}) implies that $\left\vert \nabla ^{k}\mathrm{Rm}\right\vert
^{2}S^{-2}\leq c_{k}$ on $D\left( T-2\sqrt{T}\right) \backslash D\left(
t_{0}\right) ,$ which proves (\ref{y9}).

We now complete the proof of the proposition. Define $\phi _{t}$ by%
\begin{eqnarray}
\frac{d\phi _{t}}{dt} &=&\frac{\nabla f}{\left\vert \nabla f\right\vert ^{2}}
\label{y16'} \\
\phi _{T-2\sqrt{T}} &=&\mathrm{Id}\text{ \ on }\Sigma \left( T-2\sqrt{T}%
\right) .  \notag
\end{eqnarray}%
For $q\in \Sigma \left( t_{1}\right) $ with $T-2\sqrt{T}\leq t_{1}\leq T$,
let $q_{0}\in \Sigma \left( T-2\sqrt{T}\right) $ be such that $\phi
_{t_{1}}\left( q_{0}\right) =q.$ We obtain, as in (\ref{y0}), that 
\begin{equation*}
\left\vert \frac{d}{dt}S\left( \phi _{t}\left( q_{0}\right) \right)
\right\vert \leq \frac{c}{t}.
\end{equation*}%
Integrating this from $t=T-2\sqrt{T}$ to $t=t_{1}$ implies that 
\begin{eqnarray}
S\left( q_{0}\right) &\leq &S\left( q\right) +c\ln \frac{t_{1}}{T-2\sqrt{T}}
\label{y17} \\
&\leq &S\left( q\right) +\frac{c}{\sqrt{T}}.  \notag
\end{eqnarray}%
By (\ref{CLY}) we know that $S\left( q\right) \geq \frac{c}{T}.$ Hence, (\ref%
{y17}) implies that%
\begin{equation*}
S\left( q_{0}\right) \leq c\sqrt{S}\left( q\right) .
\end{equation*}%
Since $q_{0}\in D\left( T-2\sqrt{T}\right) \backslash D\left( t_{0}\right) ,$
by (\ref{y9}) we get that 
\begin{equation}
\left\vert \nabla ^{k}\mathrm{Rm}\right\vert \left( q_{0}\right) \leq c_{k}%
\sqrt{S}\left( q\right) .  \label{y18}
\end{equation}%
Using (\ref{m8}) we compute%
\begin{eqnarray}
&&\frac{d}{dt}\left\vert \nabla ^{k}\mathrm{Rm}\right\vert ^{2}\left( \phi
_{t}\left( q_{0}\right) \right) =\frac{\left\langle \nabla \left\vert \nabla
^{k}\mathrm{Rm}\right\vert ^{2},\nabla f\right\rangle }{\left\vert \nabla
f\right\vert ^{2}}  \label{y19} \\
&=&\frac{1}{\left\vert \nabla f\right\vert ^{2}}\left( \nabla ^{k}\mathrm{Rm}%
\ast \nabla ^{k+2}\mathrm{Rm}+\nabla ^{k}\mathrm{Rm}\ast \nabla ^{k}\mathrm{%
Rm}\right)  \notag \\
&+&\frac{1}{\left\vert \nabla f\right\vert ^{2}}\left( \sum_{j=0}^{k}\nabla
^{k}\mathrm{Rm}\ast \nabla ^{j}\mathrm{Rm}\ast \nabla ^{k-j}\mathrm{Rm}%
\right) .  \notag
\end{eqnarray}%
Therefore, by (\ref{Rm})%
\begin{equation*}
\left\vert \frac{d}{dt}\left\vert \nabla ^{k}\mathrm{Rm}\right\vert \left(
\phi _{t}\left( q_{0}\right) \right) \right\vert \leq \frac{c_{k}}{t}.
\end{equation*}%
Integrating this from $t=T-2\sqrt{T}$ to $t=t_{1}$ and using (\ref{y18}) we
conclude that 
\begin{eqnarray*}
\left\vert \nabla ^{k}\mathrm{Rm}\right\vert \left( q\right) &\leq
&\left\vert \nabla ^{k}\mathrm{Rm}\right\vert \left( q_{0}\right) +c_{k}\ln 
\frac{t_{1}}{T-2\sqrt{T}} \\
&\leq &c_{k}\sqrt{S}\left( q\right) +\frac{c_{k}}{\sqrt{T}}.
\end{eqnarray*}%
Using (\ref{CLY}) again that $S\left( q\right) \geq \frac{c}{T},$ we get 
\begin{equation*}
\left\vert \nabla ^{k}\mathrm{Rm}\right\vert \left( q\right) \leq c_{k}\sqrt{%
S}\left( q\right) ,
\end{equation*}%
for any $q\in D\left( T\right) \backslash D\left( T-2\sqrt{T}\right) .$
Together with (\ref{y9}), this proves that 
\begin{equation}
\left\vert \nabla ^{k}\mathrm{Rm}\right\vert \leq c_{k}\sqrt{S}\text{ \ on }%
D\left( T\right) \backslash D\left( t_{0}\right) .  \label{y20}
\end{equation}

For any $q\in \Sigma \left( t_{1}\right) $ with $T-2\sqrt{T}\leq t_{1}\leq
T, $ let $q_{0}\in \Sigma \left( T-2\sqrt{T}\right) $ be such that $\phi
_{t_{1}}\left( q_{0}\right) =q$, where $\phi _{t}$ is defined by (\ref{y16'}%
). By (\ref{y0}) and (\ref{m8}) we have that 
\begin{eqnarray*}
\left\vert \frac{d}{dt}S\left( \phi _{t}\left( q_{0}\right) \right)
\right\vert &\leq &\frac{c}{t}\left( S+2\left\vert \mathrm{Ric}\right\vert
^{2}+\left\vert \Delta S\right\vert \right) \left( \phi _{t}\left(
q_{0}\right) \right) \\
&\leq &\frac{c}{t}\sqrt{S}\left( \phi _{t}\left( q_{0}\right) \right) ,
\end{eqnarray*}%
where for the last line we used (\ref{y20}). We rewrite this as $\left\vert 
\frac{d}{dt}\sqrt{S}\left( \phi _{t}\left( q_{0}\right) \right) \right\vert
\leq \frac{c}{t}$, and integrate in $t\in \left[ T-2\sqrt{T},t_{1}\right] .$
It follows, as for (\ref{y17}), that 
\begin{eqnarray}
\sqrt{S}\left( \phi _{t}\left( q_{0}\right) \right) &\leq &\sqrt{S}\left(
q\right) +\frac{c}{\sqrt{T}}  \label{y20'} \\
&\leq &c\sqrt{S}\left( q\right) .  \notag
\end{eqnarray}%
In particular, $S\left( q_{0}\right) \leq cS\left( q\right) $. Since $%
q_{0}\in D\left( T-2\sqrt{T}\right) \backslash D\left( t_{0}\right) ,$ by (%
\ref{y9}) we get that 
\begin{equation}
\left\vert \nabla ^{k}\mathrm{Rm}\right\vert \left( q_{0}\right) \leq
c_{k}S\left( q\right) .  \label{y21}
\end{equation}%
By (\ref{y19}), (\ref{y20}) and (\ref{y20'}) we now get 
\begin{eqnarray*}
\left\vert \frac{d}{dt}\left\vert \nabla ^{k}\mathrm{Rm}\right\vert \left(
\phi _{t}\left( q_{0}\right) \right) \right\vert &\leq &\frac{c}{t}\sqrt{S}%
\left( \phi _{t}\left( q_{0}\right) \right) \\
&\leq &\frac{c}{t}\sqrt{S}\left( q\right) .
\end{eqnarray*}%
Integrating from $t=T-2\sqrt{T}$ to $t=t_{1}$ it follows that 
\begin{eqnarray*}
\left\vert \nabla ^{k}\mathrm{Rm}\right\vert \left( q\right) &\leq
&\left\vert \nabla ^{k}\mathrm{Rm}\right\vert \left( q_{0}\right) +\frac{%
c_{k}}{\sqrt{T}}\sqrt{S}\left( q\right) \\
&\leq &c_{k}S\left( q\right) ,
\end{eqnarray*}%
where in last line we used (\ref{y21}) and (\ref{CLY}). This inequality is
true for any $q\in D\left( T\right) \backslash D\left( T-2\sqrt{T}\right) $.
Together with (\ref{y9}), it follows that $\left\vert \nabla ^{k}\mathrm{Rm}%
\right\vert \leq c_{k}S$ on $D\left( T\right) \backslash D\left(
t_{0}\right) .$ This proves the proposition.
\end{proof}

To improve Proposition \ref{D1} we use a different strategy. Let us first
record some useful consequences. Note that (\ref{m8}) and (\ref{I}) imply 
\begin{eqnarray}
\left\vert \left\langle \nabla f,\nabla ^{k}\mathrm{Rm}\right\rangle
\right\vert &\leq &\left\vert \Delta _{f}\left( \nabla ^{k}\mathrm{Rm}%
\right) \right\vert +\left\vert \Delta \left( \nabla ^{k}\mathrm{Rm}\right)
\right\vert  \label{Rm0} \\
&\leq &c_{k}S.  \notag
\end{eqnarray}%
In particular, 
\begin{equation}
\left\vert \left\langle \nabla f,\nabla S\right\rangle \right\vert \leq cS.
\label{Sc}
\end{equation}%
We can easily see from (\ref{m8}) that 
\begin{eqnarray*}
S_{ij}f_{i}f_{j} &=&\left\langle \nabla \left( S_{i}f_{i}\right) ,\nabla
f\right\rangle -f_{ij}S_{i}f_{j} \\
&=&2\left\langle \nabla \left\vert \mathrm{Ric}\right\vert ^{2},\nabla
f\right\rangle +\left\langle \nabla \left( \Delta S\right) ,\nabla
f\right\rangle \\
&&-\frac{3}{2}\left\langle \nabla S,\nabla f\right\rangle +\frac{1}{2}%
\left\vert \nabla S\right\vert ^{2}.
\end{eqnarray*}%
By (\ref{I}) and (\ref{Rm0}) it follows that 
\begin{equation}
\left\vert S_{ij}f_{i}f_{j}\right\vert \leq c\,S.  \label{Sij}
\end{equation}

We now complete the proof of Theorem \ref{Curv} by proving the following.

\begin{proposition}
\label{D2}Let $\left( M,g,f\right) $ be an $n$-dimensional, complete,
gradient shrinking Ricci soliton such that (\ref{p}) holds. Then for any $%
k\geq 0$ there exists $c_{k}>0$ such that 
\begin{equation*}
\left\vert \nabla ^{k}\mathrm{Rm}\right\vert ^{2}\leq c_{k}S^{k+2}\text{ \
on }D\left( T\right) \backslash D\left( t_{0}\right) .
\end{equation*}
\end{proposition}

\begin{proof}
For $k=0$ this follows from (\ref{I}). For the case $k=1,$ we first prove a
weaker statement that%
\begin{equation}
\left\vert \nabla \mathrm{Rm}\right\vert ^{2}\leq cS^{\frac{11}{4}}\text{ \
on }D\left( T\right) \backslash D\left( t_{0}\right) .  \label{a<3}
\end{equation}

For any $2\leq \sigma \leq 3,$ by (\ref{iw}) we have%
\begin{eqnarray*}
\Delta _{F}\left( \left\vert \nabla \mathrm{Rm}\right\vert ^{2}S^{-\sigma
}\right) &\geq& \left( 3-\sigma -\sigma \left( \frac{\sigma }{2}-1\right)
\left\vert \nabla \ln S\right\vert ^{2}\right) \left\vert \nabla \mathrm{Rm}%
\right\vert ^{2}S^{-\sigma }\\
&-& c\left\vert \mathrm{Rm}\right\vert \left\vert
\nabla \mathrm{Rm}\right\vert ^{2}S^{-\sigma }.
\end{eqnarray*}%
Hence, it follows from (\ref{I}) that 
\begin{equation*}
w:=\left\vert \nabla \mathrm{Rm}\right\vert ^{2}S^{-\sigma }
\end{equation*}
satisfies the inequality 
\begin{equation}
\Delta _{F}w\geq \left( \left( 3-\sigma \right) -cS-\sigma \left( \frac{%
\sigma }{2}-1\right) \left\vert \nabla S\right\vert ^{2}S^{-2}\right) w
\label{a2}
\end{equation}%
on $D\left( T\right) \backslash D\left( t_{0}\right) ,$ where $F:=f-\sigma
\ln S.$ We will rewrite this as an inequality on $\Sigma \left( t\right) $.
Note that 
\begin{equation}
\Delta w=\Delta _{\Sigma }w+w_{nn}+Hw_{n},  \label{Delta}
\end{equation}%
where $\Delta _{\Sigma }$ is the Laplacian on $\Sigma \left( t\right) $ and $%
H$ is the mean curvature of $\Sigma \left( t\right) .$ We first estimate $%
w_{nn}=\mathrm{Hess}(w)\left( e_{n},e_{n}\right) $ from above.

Denote by 
\begin{equation*}
u:=\left\vert \nabla \mathrm{Rm}\right\vert ^{2}
\end{equation*}
and write $w=uS^{-\sigma }.$ Then%
\begin{eqnarray}
w_{nn} &=&\frac{1}{\left\vert \nabla f\right\vert ^{2}}\left(
u_{ij}f_{i}f_{j}\right) S^{-\sigma }-\frac{2\sigma }{\left\vert \nabla
f\right\vert ^{2}}\left\langle \nabla u,\nabla f\right\rangle \left\langle
\nabla S,\nabla f\right\rangle S^{-\sigma -1}  \label{www} \\
&&+\frac{\sigma \left( \sigma +1\right) }{\left\vert \nabla f\right\vert ^{2}%
}\left\langle \nabla S,\nabla f\right\rangle ^{2}S^{-\sigma -2}u-\frac{%
\sigma }{\left\vert \nabla f\right\vert ^{2}}\left( S_{ij}f_{i}f_{j}\right)
S^{-\sigma -1}u.  \notag
\end{eqnarray}

We argue that all these terms can be bounded. Note that by (\ref{m8})%
\begin{gather}
2\left\langle \nabla f,\nabla \left( \nabla ^{k}\mathrm{Rm}\right)
\right\rangle =-\left( k+2\right) \nabla ^{k}\mathrm{Rm}  \label{n0'} \\
+\sum_{j=0}^{k}\nabla ^{j}\mathrm{Rm}\ast \nabla ^{k-j}\mathrm{Rm}+2\Delta
\nabla ^{k}\mathrm{Rm}.  \notag
\end{gather}%
Consequently, 
\begin{eqnarray}
\left\langle \nabla u,\nabla f\right\rangle &=&\nabla \mathrm{Rm}\ast \nabla 
\mathrm{Rm}+\mathrm{Rm}\ast \nabla \mathrm{Rm}\ast \nabla \mathrm{Rm}
\label{n0} \\
&&+\nabla ^{3}\mathrm{Rm}\ast \nabla \mathrm{Rm}.  \notag
\end{eqnarray}%
It can be similarly checked by using (\ref{n0'}) and (\ref{n0}) that 
\begin{gather}
\left\langle \nabla \left\langle \nabla u,\nabla f\right\rangle ,\nabla
f\right\rangle =\nabla \mathrm{Rm}\ast \nabla \mathrm{Rm}+\mathrm{Rm}\ast
\nabla \mathrm{Rm}\ast \nabla \mathrm{Rm}  \label{n} \\
+\mathrm{Rm}\ast \mathrm{Rm}\ast \nabla \mathrm{Rm}\ast \nabla \mathrm{Rm}%
+\nabla ^{2}\mathrm{Rm}\ast \nabla \mathrm{Rm}\ast \nabla \mathrm{Rm}  \notag
\\
+\nabla ^{3}\mathrm{Rm}\ast \nabla \mathrm{Rm}+\nabla ^{3}\mathrm{Rm}\ast
\nabla \mathrm{Rm}\ast \mathrm{Rm}+\nabla ^{3}\mathrm{Rm}\ast \nabla ^{3}%
\mathrm{Rm}\ast \mathrm{Rm}  \notag \\
+\nabla ^{3}\mathrm{Rm}\ast \nabla ^{3}\mathrm{Rm}+\nabla ^{5}\mathrm{Rm}%
\ast \nabla \mathrm{Rm}.  \notag
\end{gather}%
Hence, by (\ref{I}) we get $\left\vert \left\langle \nabla \left\langle
\nabla u,\nabla f\right\rangle ,\nabla f\right\rangle \right\vert \leq
c\,S^{2}.$ Moreover, (\ref{n0}) and (\ref{I}) imply 
\begin{eqnarray*}
\left\vert f_{ij}u_{i}f_{j}\right\vert &=&\left\vert \frac{1}{2}\left\langle
\nabla u,\nabla f\right\rangle -R_{ij}f_{j}u_{i}\right\vert \\
&\leq &\frac{1}{2}\left\vert \left\langle \nabla u,\nabla f\right\rangle
\right\vert +\frac{1}{2}\left\vert \left\langle \nabla S,\nabla
u\right\rangle \right\vert \\
&\leq &c\,S^{2}.
\end{eqnarray*}%
This shows%
\begin{eqnarray}
\left\vert u_{ij}f_{i}f_{j}\right\vert &\leq &\left\vert \left\langle \nabla
\left( u_{i}f_{i}\right) ,\nabla f\right\rangle \right\vert +\left\vert
f_{ij}u_{i}f_{j}\right\vert  \label{uij} \\
&\leq &cS^{2}.  \notag
\end{eqnarray}%
Since $2\leq \sigma \leq 3$, using (\ref{CLY}) we get%
\begin{equation*}
\frac{1}{\left\vert \nabla f\right\vert ^{2}}\left\vert
u_{ij}f_{i}f_{j}\right\vert S^{-\sigma }\leq c.
\end{equation*}%
Also, note that by (\ref{Sc}) and (\ref{I})%
\begin{equation*}
\frac{1}{\left\vert \nabla f\right\vert ^{2}}\left\langle \nabla S,\nabla
f\right\rangle ^{2}S^{-\sigma -2}u\leq cS^{3-\sigma }\leq c.
\end{equation*}%
Furthermore, using (\ref{n0}), one finds that 
\begin{equation*}
\frac{1}{\left\vert \nabla f\right\vert ^{2}}\left\vert \left\langle \nabla
u,\nabla f\right\rangle \right\vert \left\vert \left\langle \nabla S,\nabla
f\right\rangle \right\vert S^{-\sigma -1}\leq cS^{3-\sigma }\leq c.
\end{equation*}%
According to (\ref{Sij}), we have%
\begin{equation*}
\frac{1}{\left\vert \nabla f\right\vert ^{2}}\left\vert
S_{ij}f_{i}f_{j}\right\vert S^{-\sigma -1}u\leq cS^{3-\sigma }\leq c.
\end{equation*}%
From these estimates we get that for any $2\leq \sigma \leq 3,$%
\begin{equation}
w_{nn}\leq c.  \label{wnn}
\end{equation}%
Also, note that%
\begin{equation*}
\left\langle \nabla w,\nabla \ln S\right\rangle =\left\langle \nabla
w,\nabla \ln S\right\rangle _{\Sigma }+\frac{1}{\left\vert \nabla
f\right\vert ^{2}}\left\langle \nabla w,\nabla f\right\rangle \left\langle
\nabla \ln S,\nabla f\right\rangle .
\end{equation*}%
The last term above can be bounded by using (\ref{Sc}) together with%
\begin{eqnarray*}
\left\vert \left\langle \nabla w,\nabla f\right\rangle \right\vert &\leq
&\left\vert \left\langle \nabla u,\nabla f\right\rangle \right\vert
S^{-\sigma }+c\left\vert \left\langle \nabla \ln S,\nabla f\right\rangle
\right\vert S^{-\sigma }u \\
&\leq &cS^{-\sigma +2}.
\end{eqnarray*}%
It follows that 
\begin{equation*}
\sigma \left\langle \nabla w,\nabla \ln S\right\rangle \leq \sigma
\left\langle \nabla w,\nabla \ln S\right\rangle _{\Sigma }+c.
\end{equation*}%
Similarly, using (\ref{2ff}) we have 
\begin{equation*}
\left\vert H\,w_{n}\right\vert \leq \frac{c}{\left\vert \nabla f\right\vert
^{2}}\,\left\vert \left\langle \nabla f,\nabla w\right\rangle \right\vert
\leq c.
\end{equation*}%
Plugging this and (\ref{wnn}) into (\ref{a2}) we obtain%
\begin{eqnarray}
\Delta _{\Sigma }w &\geq &\left\langle \nabla f,\nabla w\right\rangle
-\sigma \left\langle \nabla w,\nabla \ln S\right\rangle _{\Sigma }  \label{w}
\\
&&+\left( \left( 3-\sigma \right) -cS-\sigma \left( \frac{\sigma }{2}%
-1\right) \left\vert \nabla S\right\vert ^{2}S^{-2}\right) w-c,  \notag
\end{eqnarray}%
where $w=\left\vert \nabla \mathrm{Rm}\right\vert ^{2}S^{-\sigma }.$

Let $\sigma =2+\alpha ,$ where $\alpha >0$ is to be determined later. Note
that by (\ref{I}), $\left\vert \nabla S\right\vert ^{2}S^{-2}\leq c.$ It
follows from (\ref{w}) that%
\begin{eqnarray}
\Delta _{\Sigma }w &\geq &\left\langle \nabla f,\nabla w\right\rangle
-\sigma \left\langle \nabla w,\nabla \ln S\right\rangle _{\Sigma }
\label{a3} \\
&&+\left( \frac{1}{2}-cS-c\alpha \right) w-c.  \notag
\end{eqnarray}%
Now we take $\alpha $ small so that $c\,\alpha <\frac{1}{4}.$ Then (\ref{a3}%
) becomes%
\begin{equation}
\Delta _{\Sigma }w\geq \left\langle \nabla f,\nabla w\right\rangle -\sigma
\left\langle \nabla w,\nabla \ln S\right\rangle _{\Sigma }+\left( \frac{1}{4}%
-cS\right) w-c,  \label{a4'}
\end{equation}%
where $w=\left\vert \nabla \mathrm{Rm}\right\vert ^{2}S^{-2-\alpha }.$

Let $x_{0}$ be the maximum point of $w$ in $D\left( T\right) \backslash
D\left( t_{0}\right) $. If $x_{0}\in \Sigma \left( t_{0}\right) ,$ then $%
w\left( x_{0}\right) \leq c$ by (\ref{Rm}). So we may assume that $%
x_{0}\notin \Sigma \left( t_{0}\right) .$ By maximum principle, we have $%
\Delta _{\Sigma }w\leq 0,$ $\left\langle \nabla w,\nabla \ln S\right\rangle
_{\Sigma }=0$ and $\left\langle \nabla f,\nabla w\right\rangle \geq 0$ at $%
x_{0}$. So (\ref{a4'}) implies that $\left( \frac{1}{4}-cS\left(
x_{0}\right) \right) w\left( x_{0}\right) \leq c$. If $S\left( x_{0}\right) <%
\frac{1}{8c}$, it follows that $w\left( x_{0}\right) \leq c$. On the other
hand, if $S\left( x_{0}\right) \geq \frac{1}{8c}$, then (\ref{Rm}) implies $%
w\left( x_{0}\right) \leq c.$ In conclusion, we have proved that%
\begin{equation}
\left\vert \nabla \mathrm{Rm}\right\vert ^{2}S^{-2-\alpha }\leq c\text{ \ on 
}D\left( T\right) \backslash D\left( t_{0}\right) .  \label{a5}
\end{equation}%
Using this estimate, we get from (\ref{w}) that for any $2\leq \sigma \leq 3,
$ the function 
\begin{equation*}
w:=\left\vert \nabla \mathrm{Rm}\right\vert ^{2}S^{-\sigma }
\end{equation*}%
satisfies%
\begin{equation}
\Delta _{\Sigma }w\geq \left\langle \nabla f,\nabla w\right\rangle -\sigma
\left\langle \nabla w,\nabla \ln S\right\rangle _{\Sigma }+\left( \left(
3-\sigma \right) -cS^{\alpha }\right) w-c.  \label{w'}
\end{equation}%
Let $\sigma =\frac{11}{4}.$ Then (\ref{w'}) becomes%
\begin{equation*}
\Delta _{\Sigma }w\geq \left\langle \nabla f,\nabla w\right\rangle -\sigma
\left\langle \nabla w,\nabla \ln S\right\rangle _{\Sigma }+\left( \frac{1}{4}%
-cS^{\alpha }\right) w-c.
\end{equation*}%
Applying the maximum principle as in the proof of (\ref{a5}), one concludes
that $w$ is bounded on $D\left( T\right) \backslash D\left( t_{0}\right) .$
This shows that%
\begin{equation}
\left\vert \nabla \mathrm{Rm}\right\vert ^{2}\leq cS^{\frac{11}{4}}\text{ \
on }D\left( T\right) \backslash D\left( t_{0}\right)   \label{a6}
\end{equation}%
and (\ref{a<3}) is established.

We now prove that for any $k\geq p\geq 1$ there exists a constant $c_{k,p},$
depending on $k$ and $p,$ such that 
\begin{equation}
\left\vert \nabla ^{k}\mathrm{Rm}\right\vert ^{2}\leq c_{k,p}\,S^{p+1}\text{
\ on }D\left( T\right) \backslash D\left( t_{0}\right) .  \label{i1}
\end{equation}

The proof is by induction on $p.$ By (\ref{I}), clearly (\ref{i1}) is true
for $p=1.$ Now assume that (\ref{i1}) holds for $p=1,2,\cdots ,l.$ We will
prove it for $p=l+1.$ That is, if%
\begin{equation}
\left\vert \nabla ^{j}\mathrm{Rm}\right\vert ^{2}\leq c_{j}\,S^{j+1}\text{ \
for all \ }j\leq l  \label{i2}
\end{equation}%
and 
\begin{equation}
\left\vert \nabla ^{j}\mathrm{Rm}\right\vert ^{2}\leq c_{j,l}\,S^{l+1}\text{
\ \ for all \ \ }j>l,  \label{i3}
\end{equation}%
then%
\begin{equation}
\left\vert \nabla ^{k}\mathrm{Rm}\right\vert ^{2}\leq c_{k,l+1}\,S^{l+2}%
\text{\ for all }k\geq l+1.  \label{i4}
\end{equation}

For any $\sigma >0,$ we have by (\ref{iw1}) that 
\begin{align}
\Delta _{f}\left( \left\vert \nabla ^{k}\mathrm{Rm}\right\vert
^{2}S^{-\sigma }\right) & \geq \left( \left( k+2\right) -\sigma -c\sigma
^{2}\left\vert \nabla \ln S\right\vert ^{2}\right) \left\vert \nabla ^{k}%
\mathrm{Rm}\right\vert ^{2}S^{-\sigma }  \label{i4'} \\
& -c\Sigma _{j=0}^{k}\left\vert \nabla ^{j}\mathrm{Rm}\right\vert \left\vert
\nabla ^{k-j}\mathrm{Rm}\right\vert \left\vert \nabla ^{k}\mathrm{Rm}%
\right\vert S^{-\sigma }.  \notag
\end{align}%
Note that (\ref{i2}) and (\ref{i3}) imply that for $k\geq l+1,$ 
\begin{equation*}
\left\vert \nabla ^{j}\mathrm{Rm}\right\vert \left\vert \nabla ^{k-j}\mathrm{%
Rm}\right\vert \left\vert \nabla ^{k}\mathrm{Rm}\right\vert S^{-l-2}\leq
c_{k,l+1},
\end{equation*}%
where $c_{k,l+1}$ is a constant depending on $c_{j}$ from (\ref{i2}) for $%
j\leq l$, and on $c_{h,l}$ from (\ref{i3}) for $l<h\leq k$. It now follows
from (\ref{a6}) and (\ref{i4'}) that for any $k\geq l+1,$ by letting $\sigma
=l+2,$ the function 
\begin{equation*}
w:=\left\vert \nabla ^{k}\mathrm{Rm}\right\vert ^{2}S^{-l-2}
\end{equation*}
satisfies the inequality%
\begin{equation}
\Delta _{f}w\geq \left( 1-c_{k}\sqrt{S}\right) w-c_{k,l+1}.  \label{i5}
\end{equation}%
Now we follow a similar strategy as in the proof of (\ref{a6}) to show that $%
w_{nn}\leq c_{k,l+1}.$ Denote by 
\begin{equation*}
u:=\left\vert \nabla ^{k}\mathrm{Rm}\right\vert ^{2},
\end{equation*}
so that $w=uS^{-l-2}.$ Note that by (\ref{Sc}) and (\ref{i3}), 
\begin{equation*}
\frac{1}{\left\vert \nabla f\right\vert ^{2}}\left\langle \nabla S,\nabla
f\right\rangle ^{2}S^{-l-4}u\leq cS^{-l-1}u\leq c_{k,l+1}
\end{equation*}%
and by (\ref{Sij}) and (\ref{i3}),%
\begin{equation*}
\frac{1}{\left\vert \nabla f\right\vert ^{2}}\left\vert
S_{ij}f_{i}f_{j}\right\vert S^{-l-3}\,u\leq c\,S^{-l-1}u\leq c_{k,l+1}.
\end{equation*}%
Furthermore, according to (\ref{m8}) we have 
\begin{eqnarray}
\left\langle \nabla u,\nabla f\right\rangle &=&\nabla ^{k}\mathrm{Rm}\ast
\nabla ^{k}\mathrm{Rm}+\nabla ^{k}\mathrm{Rm}\ast \nabla ^{k+2}\mathrm{Rm}
\label{i6} \\
&&+\Sigma _{j=0}^{k}\nabla ^{j}\mathrm{Rm}\ast \nabla ^{k-j}\mathrm{Rm}\ast
\nabla ^{k}\mathrm{Rm}.  \notag
\end{eqnarray}%
It follows immediately from (\ref{i2}) and (\ref{i3}) that $\left\vert
\left\langle \nabla u,\nabla f\right\rangle \right\vert \leq
c_{k,l+1}\,S^{l+1}$, where $c_{k,l+1}$ depends on $c_{j}$ from (\ref{i2})
for $j\leq l$ and on $c_{h,l}$ from (\ref{i3}) for $l<h\leq k+2$. Hence,
this proves that 
\begin{equation*}
\frac{1}{\left\vert \nabla f\right\vert ^{2}}\left\vert \left\langle \nabla
u,\nabla f\right\rangle \right\vert \left\vert \left\langle \nabla S,\nabla
f\right\rangle \right\vert S^{-l-3}\leq c_{k,l+1}.
\end{equation*}%
Similarly, we have $\left\vert \left\langle \nabla \left\langle \nabla
u,\nabla f\right\rangle ,\nabla f\right\rangle \right\vert \leq
c_{k,l+1}\,S^{l+1}$. As in (\ref{uij}) we get 
\begin{equation*}
\left\vert u_{ij}f_{i}f_{j}\right\vert S^{-l-1}\leq c_{k,l+1}.
\end{equation*}

The above estimates, together with (\ref{www}), imply%
\begin{equation}
w_{nn}\leq c_{k,l+1}.  \label{i7'}
\end{equation}%
Finally, we also get from above and from (\ref{2ff}) that 
\begin{eqnarray*}
Hw_{n} &\leq &\frac{c}{\left\vert \nabla f\right\vert ^{2}}\left\vert
\left\langle \nabla w,\nabla f\right\rangle \right\vert \\
&\leq &\frac{c}{\left\vert \nabla f\right\vert ^{2}}\left( \left\langle
\nabla u,\nabla f\right\rangle S^{-l-2}+(l+2)\,\left\vert \left\langle
\nabla \ln S,\nabla f\right\rangle \right\vert S^{-l-2}u\right) \\
&\leq &c_{k,l+1}.
\end{eqnarray*}

It is easy to see from (\ref{i5}) and (\ref{Delta}) that 
\begin{equation*}
w:=\left\vert \nabla ^{k}\mathrm{Rm}\right\vert ^{2}S^{-l-2}
\end{equation*}
satisfies%
\begin{equation}
\Delta _{\Sigma }w\geq \left\langle \nabla w,\nabla f\right\rangle +\left(
1-c_{k}\sqrt{S}\right) w-c_{k,l+1}  \label{i8}
\end{equation}%
for any $k\geq l+1.$ By the maximum principle, (\ref{i8}) implies that $%
w\leq c_{k,l+1}$ on $D\left( T\right) \backslash D\left( t_{0}\right) $ for
all $k\geq l+1.$ This proves (\ref{i4}) and completes the induction step.

In conclusion, we have established (\ref{i1}). In particular, for all $p\geq
1,$ there exists $c_{p}>0$ such that 
\begin{equation}
\left\vert \nabla ^{p}\mathrm{Rm}\right\vert ^{2}\leq c_{p}\,S^{p+1}\text{ \
on }D\left( T\right) \backslash D\left( t_{0}\right) .  \label{i9}
\end{equation}

We are now ready to prove an estimate like (\ref{a<3}) for all $k\geq 1,$
that is,%
\begin{equation}
\left\vert \nabla ^{k}\mathrm{Rm}\right\vert ^{2}\leq c_{k}\,S^{k+\frac{7}{4}%
}\text{ \ on }D\left( T\right) \backslash D\left( t_{0}\right) .  \label{ak}
\end{equation}%
Note that (\ref{i9}) implies%
\begin{equation*}
\Sigma _{j=0}^{k}\left\vert \nabla ^{j}\mathrm{Rm}\right\vert \left\vert
\nabla ^{k-j}\mathrm{Rm}\right\vert \leq c_{k}\,S^{\frac{k}{2}+1}.
\end{equation*}%
So from (\ref{i4'}) we get%
\begin{eqnarray*}
\Delta _{f}\left( \left\vert \nabla ^{k}\mathrm{Rm}\right\vert ^{2}S^{-k-%
\frac{7}{4}}\right) &\geq &\left( \frac{1}{4}-c_{k}\,\sqrt{S}\right)
\left\vert \nabla ^{k}\mathrm{Rm}\right\vert ^{2}\,S^{-k-\frac{7}{4}%
}-c_{k}\left\vert \nabla ^{k}\mathrm{Rm}\right\vert S^{-\frac{k}{2}-\frac{3}{%
4}} \\
&\geq &\left( \frac{1}{6}-c_{k}\sqrt{S}\right) \left\vert \nabla ^{k}\mathrm{%
Rm}\right\vert ^{2}S^{-k-\frac{7}{4}}-c_{k}.
\end{eqnarray*}%
Hence the function 
\begin{equation*}
w:=\left\vert \nabla ^{k}\mathrm{Rm}\right\vert ^{2}S^{-k-\frac{7}{4}}
\end{equation*}
satisfies%
\begin{equation}
w_{nn}+Hw_{n}+\Delta _{\Sigma }w\geq \left\langle \nabla f,\nabla
w\right\rangle +\left( \frac{1}{6}-c_{k}\sqrt{S}\right) w-c_{k}.  \label{w1}
\end{equation}%
Following the proof of (\ref{i7'}) it can be seen that $w_{nn}+Hw_{n}\leq
c_{k}.$ Therefore, by applying the maximum principle to (\ref{w1}), we have $%
w\leq c_{k}$ on $D\left( T\right) \backslash D\left( t_{0}\right) .$ This
shows that (\ref{ak}) is indeed true.

We now finish the proof of the proposition by showing%
\begin{equation}
\left\vert \nabla ^{k}\mathrm{Rm}\right\vert ^{2}\leq c_{k}\,S^{k+2}\text{ \
on }D\left( T\right) \backslash D\left( t_{0}\right)  \label{k}
\end{equation}%
for each $k\geq 1.$

Let 
\begin{equation*}
w:=\left\vert \nabla ^{k}\mathrm{Rm}\right\vert ^{2}S^{-k-2}.
\end{equation*}
Using (\ref{i4'}) and (\ref{ak}), we get%
\begin{align}
\Delta _{f}w& \geq -c_{k}\left\vert \nabla \ln S\right\vert ^{2}w-c\Sigma
_{j=0}^{k}\left\vert \nabla ^{j}\mathrm{Rm}\right\vert \left\vert \nabla
^{k-j}\mathrm{Rm}\right\vert \left\vert \nabla ^{k}\mathrm{Rm}\right\vert
S^{-k-2}  \label{k1} \\
& \geq -c_{k}S^{\frac{1}{2}}.  \notag
\end{align}%
On the other hand, it is easy to check that 
\begin{eqnarray}
\Delta _{f}S^{\frac{1}{2}} &=&\frac{1}{2}\left( \Delta _{f}S\right) S^{-%
\frac{1}{2}}-\frac{1}{4}\left\vert \nabla S\right\vert ^{2}S^{-\frac{3}{2}}
\label{k2} \\
&\geq &\frac{1}{2}S^{\frac{1}{2}}\left( 1-cS^{\frac{3}{4}}\right) .  \notag
\end{eqnarray}

From (\ref{k1}) and (\ref{k2}) we see that there exists a constant $C_{k}>0$
such that $\upsilon :=w+C_{k}S^{\frac{1}{2}}$ satisfies 
\begin{equation}
\Delta _{f}\upsilon \geq S^{\frac{1}{2}}\left( 1-c_{k}S^{\frac{3}{4}}\right)
.  \label{k3}
\end{equation}%
We now bound $\upsilon _{nn}$ from above. By (\ref{Sij}), we have%
\begin{equation}
\left( S^{\frac{1}{2}}\right) _{nn}\leq \frac{1}{2}\frac{S_{ij}f_{i}f_{j}}{%
\left\vert \nabla f\right\vert ^{2}}S^{-\frac{1}{2}}\leq cS^{\frac{3}{2}}.
\label{k4}
\end{equation}%
To bound $w_{nn},$ we use (\ref{www}). By (\ref{ak}), we get%
\begin{eqnarray*}
\frac{1}{\left\vert \nabla f\right\vert ^{2}}\left\langle \nabla S,\nabla
f\right\rangle ^{2}S^{-k-4}u &\leq &cS^{-k-1}u \\
&\leq &c_{k}S^{\frac{3}{4}}.
\end{eqnarray*}%
Also, by (\ref{Sij}) and (\ref{ak}),%
\begin{eqnarray*}
\frac{1}{\left\vert \nabla f\right\vert ^{2}}\left\vert
S_{ij}f_{i}f_{j}\right\vert S^{-k-3}u &\leq &cS^{-k-1}u \\
&\leq &c_{k}S^{\frac{3}{4}}.
\end{eqnarray*}%
Furthermore, using (\ref{i6}) we see that $\left\vert \left\langle \nabla
u,\nabla f\right\rangle \right\vert \leq c_{k}S^{k+\frac{7}{4}}$, hence%
\begin{equation*}
\frac{1}{\left\vert \nabla f\right\vert ^{2}}\left\vert \left\langle \nabla
u,\nabla f\right\rangle \right\vert \left\vert \left\langle \nabla S,\nabla
f\right\rangle \right\vert S^{-k-3}\leq c_{k}S^{\frac{3}{4}}.
\end{equation*}%
In a similar way it can be shown that 
\begin{equation*}
\frac{1}{\left\vert \nabla f\right\vert ^{2}}\left\vert
u_{ij}f_{i}f_{j}\right\vert S^{-k-2}\leq c_{k}S^{\frac{3}{4}}.
\end{equation*}%
Combining all these estimates implies that%
\begin{equation}
w_{nn}\leq c_{k}S^{\frac{3}{4}}.  \label{k5}
\end{equation}%
From (\ref{k3}), (\ref{k4}) and (\ref{k5}) it can be easily seen that 
\begin{equation*}
\Delta _{\Sigma }\upsilon \geq \left\langle \nabla \upsilon ,\nabla
f\right\rangle +S^{\frac{1}{2}}\left( 1-c_{k}S^{\frac{1}{4}}\right) .
\end{equation*}%
Therefore, if the maximum of $\upsilon $ does not occur on $\Sigma \left(
t_{0}\right) ,$ then $S^{\frac{1}{4}}\geq \frac{1}{c_{k}}$ at the maximum
point. By (\ref{Rm}), we have $\upsilon \leq c_{k}$ on $D\left( T\right)
\backslash D\left( t_{0}\right) $ and%
\begin{equation*}
\left\vert \nabla ^{k}\mathrm{Rm}\right\vert ^{2}\leq c_{k}S^{k+2}\text{ \
on }D\left( T\right) \backslash D\left( t_{0}\right) .
\end{equation*}%
This proves the proposition.
\end{proof}

Proposition \ref{D2} allows us to establish the following Harnack estimate
for the scalar curvature. Assume that (\ref{p}) holds. For $t\geq t_{0},$
define $\phi _{t}$ as follows.

\begin{eqnarray}
\frac{d\phi _{t}}{dt} &=&\frac{\nabla f}{\left\vert \nabla f\right\vert ^{2}}
\label{S0} \\
\phi _{t_{0}} &=&\mathrm{Id}\text{ \ on }\Sigma \left( t_{0}\right) .  \notag
\end{eqnarray}

For $x\in \Sigma \left( t_{0}\right),$ let $S\left( t\right) :=S\left( \phi
_{t}\left( x\right) \right),$ where $t_{0}\leq t\leq T.$ Then

\begin{eqnarray*}
\frac{dS}{dt} &=&\frac{\left\langle \nabla S,\nabla f\right\rangle }{%
\left\vert \nabla f\right\vert ^{2}} \\
&=&\frac{\Delta S-S+2\left\vert \mathrm{Ric}\right\vert ^{2}}{t-S}.
\end{eqnarray*}%
Using the estimate $\left\vert \Delta S\right\vert \leq cS^{2}$ from
Proposition \ref{D2}, we get

\begin{equation*}
\left\vert \frac{dS}{dt}+\frac{S}{t}\right\vert \leq C_{1}\frac{S^{2}}{t}
\end{equation*}%
for some constant $C_{1}>0.$ This can be rewritten into 
\begin{equation*}
\left\vert \frac{\left( tS\right) ^{\prime }}{\left( tS\right) ^{2}}%
\right\vert \leq \frac{C_{1}}{t^{2}}.
\end{equation*}%
Integrating in $t$ gives 
\begin{equation}
\left\vert \frac{1}{t_{2}S\left( t_{2}\right) }-\frac{1}{t_{1}S\left(
t_{1}\right) }\right\vert \leq C_{1}\,\left( \frac{1}{t_{1}}-\frac{1}{t_{2}}%
\right)  \label{S}
\end{equation}%
for any $t_{1}$ and $t_{2}$ with $t_{0}<t_{1}<t_{2}<T.$ Hence, if there
exists $t_{0}<t_{1}<T$ with $S\left( t_{1}\right) \leq \frac{1}{2C_{1}},$
then 
\begin{equation}
S\left( t\right) \leq \frac{1}{C_{1}}\frac{t_{1}}{t}\text{ \ for all }%
t_{1}\leq t\leq T.  \label{S'}
\end{equation}

\section{\label{Convergence}Ricci shrinkers asymptotic to round cylinder}

In this section, we use the estimates from section \ref{Curvature} to prove
Theorem \ref{cyl}. We continue to denote by $M$ an $n$-dimensional,
complete, gradient shrinking Ricci soliton with bounded curvature, and by $%
\left\{ e_{1},e_{2},\cdots, e_{n}\right\}$ a local orthonormal frame with 
\begin{equation*}
e_{n}:=\frac{\nabla f}{\left\vert \nabla f\right\vert }.
\end{equation*}%
As before, the indices $a,b, c, d=1,2,\cdots, n-1$ and $i, j, k, l=1,2,
\cdots, n.$

In the following, let us assume that (\ref{p}) hold on $D(T)\backslash D(t_0)$. 
By Theorem \ref{Curv} and (\ref{d1}) we have 
\begin{eqnarray}
\left\vert R_{ijkn}\right\vert &\leq &\frac{2\left\vert \nabla \mathrm{Ric}%
\right\vert }{\left\vert \nabla f\right\vert }  \label{a} \\
&\leq &cS^{\frac{3}{2}}f^{-\frac{1}{2}}.  \notag
\end{eqnarray}%
Using (\ref{m8}) we get

\begin{eqnarray*}
\left\vert R_{inkn}\right\vert  &=&\frac{\left\vert
R_{ijkl}f_{j}f_{l}\right\vert }{\left\vert \nabla f\right\vert ^{2}} \\
&=&\frac{\left\vert f_{j}\nabla _{j}R_{ik}-f_{j}\nabla _{i}R_{jk}\right\vert 
}{\left\vert \nabla f\right\vert ^{2}} \\
&\leq &\frac{\left\vert \left\langle \nabla f,\nabla R_{ik}\right\rangle
\right\vert +\left\vert R_{jk}f_{ij}-\nabla _{i}\left( R_{jk}f_{j}\right)
\right\vert }{\left\vert \nabla f\right\vert ^{2}}.
\end{eqnarray*}%
Since 
\begin{equation*}
\left\langle \nabla f,\nabla R_{ik}\right\rangle =\Delta
R_{ik}-R_{ik}+2R_{ijkl}R_{jl},
\end{equation*}%
we obtain from Theorem \ref{Curv} that 
\begin{equation*}
\left\vert \left\langle \nabla f,\nabla R_{ik}\right\rangle \right\vert \leq
cS.
\end{equation*}%
Similarly, we have 
\begin{eqnarray*}
\left\vert R_{jk}f_{ij}-\nabla _{i}\left( R_{jk}f_{j}\right) \right\vert 
&=&\left\vert \frac{1}{2}R_{ik}-R_{jk}R_{ij}-\frac{1}{2}\nabla _{i}\nabla
_{k}S\right\vert  \\
&\leq &cS.
\end{eqnarray*}%
In conclusion, we get from above that 
\begin{equation}
\left\vert R_{inkn}\right\vert \leq cSf^{-1}.  \label{a'}
\end{equation}

Consequently, for $U,V,W$ defined in (\ref{U,V,W}), we have on $D(T)\backslash D(t_0)$,

\begin{equation}
\left\langle U,V\right\rangle =O\left( S^{2}f^{-1}\right), \ \left\langle
U,W\right\rangle =O\left( S^{2}f^{-1}\right), \ \left\langle
V,W\right\rangle =O\left( S^{2}f^{-1}\right)  \label{c4}
\end{equation}%
and 
\begin{eqnarray}
\left\vert U\right\vert ^{2} &=&\frac{2}{\left( n-1\right) \left( n-2\right) 
}S^{2},  \label{c5} \\
\left\vert V\right\vert ^{2} &=&\frac{4}{n-3}\left\vert \overset{\circ }{%
\mathrm{Ric}_{\Sigma }}\right\vert ^{2}+O(S^{2}f^{-1}),  \notag \\
\left\vert \mathrm{Rm}\right\vert ^{2} &=&\left\vert \overset{\circ }{%
\mathrm{Rm}}_{\Sigma }\right\vert ^{2}+\frac{2}{\left( n-1\right) \left(
n-2\right) }S^{2}+O\left( S^{2}f^{-1}\right),  \notag \\
\left\vert \mathrm{Ric}\right\vert ^{2} &=&\left\vert \overset{\circ }{%
\mathrm{Ric}_{\Sigma }}\right\vert ^{2}+\frac{1}{n-1}S^{2}+O\left(
S^{2}f^{-1}\right),  \notag \\
\left\vert \mathrm{Rm}\right\vert ^{2} &=&\left\vert U\right\vert
^{2}+\left\vert V\right\vert ^{2}+\left\vert W\right\vert ^{2}+O\left(
S^{2}f^{-1}\right).  \notag
\end{eqnarray}
Here and below, the constants implicit in the big $O$ notation depend only on $n,\eta_1, \eta_2, C_0$ 
and $A_0,...,A_{Kk}$, as specified in Theorem  \ref{Curv}, so they are independent of $t_0$ and $T$.  

We restate Theorem \ref{cyl} here. Without loss of generality, we assume $M$
has only one end.

\begin{theorem}
\label{cylinder}Let $\left( M,g,f\right) $ be an $n$-dimensional, complete,
gradient shrinking Ricci soliton with $\left\vert \mathrm{Rm}\right\vert
\leq C.$ Assume that there exists a sequence of points $x_{k}\rightarrow
\infty $ such that $\left( M,g,x_{k}\right) $ converges to a round cylinder $%
\mathbb{R}\times \mathbb{S}^{n-1}/\Gamma.$ Then $M$ is smoothly asymptotic
to the same round cylinder.
\end{theorem}

\begin{proof}
Without loss of generality, we may assume that $x_{k}$ converges to a point
in $\{0\}\times \mathbb{S}^{n-1}/\Gamma .$ We first claim that $\Sigma
(t_{k}),$ the level set of $f$ containing $x_{k},$ must converge to $%
\{0\}\times \mathbb{S}^{n-1}/\Gamma $. Indeed, consider the vector field
defined on $M\backslash D\left( t_{0}\right) ,$ 
\begin{equation*}
X:=\frac{\nabla f}{\left\vert \nabla f\right\vert }.
\end{equation*}%
Since $\left\vert X\right\vert =1$, we get that $X$ converges smoothly to a
vector field $X_{\infty }$ on $\mathbb{R}\times \mathbb{S}^{n-1}/\Gamma $.
It is easy to see that $X_{\infty }$ is in fact parallel, because 
\begin{eqnarray*}
\left\vert \nabla X\right\vert  &\leq &2\frac{\left\vert \mathrm{Hess}\left(
f\right) \right\vert }{\left\vert \nabla f\right\vert } \\
&\leq &cf^{-\frac{1}{2}}.
\end{eqnarray*}%
This proves that $X_{\infty }$ is the radial vector on $\mathbb{R}\times 
\mathbb{S}^{n-1}/\Gamma $ , and hence the level set corresponding to $x_{k}$
converges to $\{0\}\times \mathbb{S}^{n-1}/\Gamma $. In particular, it
follows that for any $\varepsilon >0$ there exists sufficiently large $t_{0}$
such that

\begin{eqnarray}
\sup_{\Sigma \left( t_{0}\right) }\left\vert \overset{\circ }{\mathrm{Rm}%
_{\Sigma }}\right\vert ^{2}S^{-2} &<&\frac{\varepsilon }{2},  \label{c12} \\
\sup_{\Sigma \left( t_{0}\right) }\left\vert S-\frac{n-1}{2}\right\vert
&<&\varepsilon.  \notag
\end{eqnarray}

\begin{claim}
\label{Conv}For $\varepsilon >0$ and $t_{0}>0$ such that (\ref{c12}) holds
we have 
\begin{equation}
\sup_{\Sigma \left( t\right) }\left\vert \overset{\circ }{\mathrm{Rm}%
_{\Sigma }}\right\vert ^{2}S^{-2}<\varepsilon \text{ \ \ for all } t\geq
t_{0}.  \label{c13}
\end{equation}
\end{claim}

To prove Claim \ref{Conv}, let 
\begin{equation}
T:=\sup \left\{ t:\sup_{\Sigma \left( r\right) }\left\vert \overset{\circ }{%
\mathrm{Rm}_{\Sigma }}\right\vert ^{2}S^{-2}<\varepsilon \text{ \ for all }%
t_{0}\leq r\leq t\right\}.  \label{c14}
\end{equation}%
If $T<\infty,$ then%
\begin{equation}
\sup_{\Sigma \left( T\right) }\left\vert \overset{\circ }{\mathrm{Rm}%
_{\Sigma }}\right\vert ^{2}S^{-2}=\varepsilon.  \label{c14'}
\end{equation}
Note that (\ref{c14}) and (\ref{c12}) imply that (\ref{p}) holds on $D(T)\backslash D(t_0)$, for $\eta_1$ and $\eta_2$ depending only on dimension. Hence, (\ref{c4}) and (\ref{c5}) hold on  $D(T)\backslash D(t_0)$ as well. 

We have the following formula (see Ch. 2.7 in \cite{CLN} or \cite{PW})%
\begin{equation}
\Delta _{f}\left\vert \mathrm{Rm}\right\vert ^{2}=2\left\vert \nabla \mathrm{%
Rm}\right\vert ^{2}+2\left\vert \mathrm{Rm}\right\vert
^{2}-8R_{ijkl}R_{piqk}R_{pjql}-2R_{ijkl}R_{ijpq}R_{pqkl}.  \label{c1}
\end{equation}%
For the function $G$ given by%
\begin{equation*}
G:=\left\vert \mathrm{Rm}\right\vert ^{2}S^{-2}-\frac{2}{\left( n-1\right)
\left( n-2\right) },
\end{equation*}%
using (\ref{c1}) and arguing as in (\ref{iw}) we obtain the following
inequality (cf. Lemma 3.2 in \cite{Hu}).%
\begin{equation}
\Delta _{f}G\geq -2\left\langle \nabla G,\nabla \ln S\right\rangle +4S^{-3}P,
\label{c7}
\end{equation}%
where%
\begin{equation}
P:=-2SR_{ijkl}R_{piqk}R_{pjql}-\frac{1}{2}SR_{ijkl}R_{ijpq}R_{pqkl}+\left%
\vert \mathrm{Rm}\right\vert ^{2}\left\vert \mathrm{Ric}\right\vert ^{2}.
\label{c8}
\end{equation}%
Note that by (\ref{c5}),%
\begin{equation}
G=\left\vert \overset{\circ }{\mathrm{Rm}_{\Sigma }}\right\vert
^{2}S^{-2}+O\left( f^{-1}\right) .  \label{G}
\end{equation}%
By (\ref{c4}) and (\ref{c5}) we get that 
\begin{align}
& \left\vert \mathrm{Rm}\right\vert ^{2}\left\vert \mathrm{Ric}\right\vert
^{2}  \label{c9} \\
& =\left( \left\vert U\right\vert ^{2}+\left\vert V\right\vert
^{2}+\left\vert W\right\vert ^{2}\right) \left( \left\vert \overset{\circ }{%
\mathrm{Ric}_{\Sigma }}\right\vert ^{2}+\frac{1}{n-1}S^{2}\right) +O\left(
S^{4}f^{-1}\right)   \notag \\
& =\left\vert W\right\vert ^{2}\left\vert \overset{\circ }{\mathrm{Ric}%
_{\Sigma }}\right\vert ^{2}+\frac{1}{n-1}S^{2}\left\vert W\right\vert ^{2}+%
\frac{4}{n-3}\left\vert \overset{\circ }{\mathrm{Ric}_{\Sigma }}\right\vert
^{4}  \notag \\
& +\frac{2\left( 3n-7\right) }{\left( n-1\right) \left( n-2\right) \left(
n-3\right) }S^{2}\left\vert \overset{\circ }{\mathrm{Ric}_{\Sigma }}%
\right\vert ^{2}  \notag \\
& +\frac{2}{\left( n-1\right) ^{2}\left( n-2\right) }S^{4}+O\left(
S^{4}f^{-1}\right) .  \notag
\end{align}%
A much longer computation of similar nature implies (see Theorem 3.3 in \cite%
{Hu})%
\begin{eqnarray}
&&-2SR_{ijkl}R_{piqk}R_{pjql}-\frac{1}{2}SR_{ijkl}R_{ijpq}R_{pqkl}
\label{c10} \\
&=&-2SR_{abcd}R_{eagc}R_{ebgd}-\frac{1}{2}SR_{abcd}R_{abeg}R_{egcd}+O\left(
S^{4}f^{-1}\right)   \notag \\
&=&-2SW_{abcd}W_{eagc}W_{ebgd}-\frac{1}{2}SW_{abcd}W_{abeg}W_{egcd}  \notag
\\
&&-\frac{6}{n-3}SW_{abcd}\overset{\circ }{R}_{ac}\overset{\circ }{R}_{bd}-%
\frac{6}{\left( n-1\right) \left( n-2\right) }S^{2}\left\vert \overset{\circ 
}{\mathrm{Ric}_{\Sigma }}\right\vert ^{2}  \notag \\
&&+\frac{8}{\left( n-3\right) ^{2}}S\overset{\circ }{R}_{ab}\overset{\circ }{%
R}_{bc}\overset{\circ }{R}_{ac}-\frac{2}{\left( n-1\right) ^{2}\left(
n-2\right) }S^{4}+O\left( S^{4}f^{-1}\right) .  \notag
\end{eqnarray}%
From (\ref{c8}), (\ref{c9}) and (\ref{c10}) we conclude that

\begin{eqnarray}
P &=&\frac{1}{n-1}S^{2}\left\vert W\right\vert
^{2}-2SW_{abcd}W_{eafc}W_{ebfd}  \label{c11} \\
&&-\frac{1}{2}SW_{abcd}W_{abef}W_{efcd}  \notag \\
&&+\frac{4}{\left( n-1\right) \left( n-2\right) \left( n-3\right) }%
S^{2}\left\vert \overset{\circ }{\mathrm{Ric}_{\Sigma }}\right\vert ^{2}+%
\frac{4}{n-3}\left\vert \overset{\circ }{\mathrm{Ric}_{\Sigma }}\right\vert
^{4}  \notag \\
&&+\frac{8}{\left( n-3\right) ^{2}}S\overset{\circ }{R}_{ab}\overset{\circ }{%
R}_{bc}\overset{\circ }{R}_{ac}+\left\vert W\right\vert ^{2}\left\vert 
\overset{\circ }{\mathrm{Ric}_{\Sigma }}\right\vert ^{2}  \notag \\
&&-\frac{6}{n-3}SW_{abcd}\overset{\circ }{R}_{ac}\overset{\circ }{R}%
_{bd}+O\left( S^{4}f^{-1}\right) .  \notag
\end{eqnarray}%
Therefore, we have%
\begin{eqnarray*}
P &\geq &\frac{1}{n-1}S^{2}\left\vert W\right\vert ^{2}+\frac{4}{\left(
n-1\right) \left( n-2\right) \left( n-3\right) }S^{2}\left\vert \overset{%
\circ }{\mathrm{Ric}_{\Sigma }}\right\vert ^{2} \\
&&-\frac{5}{2}S\left\vert W\right\vert ^{3}-\frac{8}{\left( n-3\right) ^{2}}%
S\left\vert \overset{\circ }{\mathrm{Ric}_{\Sigma }}\right\vert ^{3} \\
&&-\frac{6}{n-3}S\left\vert W\right\vert \left\vert \overset{\circ }{\mathrm{%
Ric}_{\Sigma }}\right\vert ^{2}-cS^{4}f^{-1}.
\end{eqnarray*}%
Since by (\ref{c14}) for all $t_{0}\leq t\leq T,$%
\begin{eqnarray}
\left\vert W\right\vert ^{2}+\frac{4}{n-3}\left\vert \overset{\circ }{%
\mathrm{Ric}_{\Sigma }}\right\vert ^{2} &\leq &\left\vert \overset{\circ }{%
\mathrm{Rm}_{\Sigma }}\right\vert ^{2}+cS^{2}f^{-1}  \label{pinching} \\
&\leq &\varepsilon S^{2}+cS^{2}f^{-1}  \notag \\
&\leq &2\varepsilon S^{2},  \notag
\end{eqnarray}%
it follows that 
\begin{eqnarray*}
P &\geq &\left( \frac{1}{n-1}-c\sqrt{\varepsilon }\right) S^{2}\left\vert
W\right\vert ^{2} \\
&&+\left( \frac{4}{\left( n-1\right) \left( n-2\right) \left( n-3\right) }-c%
\sqrt{\varepsilon }\right) S^{2}\left\vert \overset{\circ }{\mathrm{Ric}%
_{\Sigma }}\right\vert ^{2} \\
&&-c\,S^{4}\,f^{-1}
\end{eqnarray*}%
for some constant $c>0$ depending only on $n.$ In particular, there exists $%
\theta >0$ depending only on $n$ such that%
\begin{equation*}
P\geq \theta S^{2}\left\vert \overset{\circ }{\mathrm{Rm}_{\Sigma }}%
\right\vert ^{2}-c\,S^{4}\,f^{-1}.
\end{equation*}%
As $\left\vert \overset{\circ }{\mathrm{Rm}_{\Sigma }}\right\vert ^{2}\geq
GS^{2}-cS^{2}f^{-1},$ it follows from (\ref{c7}) that on $D\left( T\right)
\backslash D\left( t_{0}\right) $%
\begin{equation}
\Delta G\geq \left\langle \nabla G,\nabla f\right\rangle -2\left\langle
\nabla G,\nabla \ln S\right\rangle +\theta SG-cSf^{-1}.  \label{c15}
\end{equation}%
Note that 
\begin{equation*}
\Delta G=\Delta _{\Sigma }G+G_{nn}+HG_{n},
\end{equation*}%
where $\Delta _{\Sigma }$ is the Laplacian on $\Sigma \left( t\right) .$ We
now bound $G_{nn}$ by a similar argument as in the proof of Proposition \ref%
{D2}. Let $u:=\left\vert \mathrm{Rm}\right\vert ^{2}$ and $w:=\frac{u}{S^{2}}%
.$ Then $G_{nn}=w_{nn}.$ By (\ref{www}) we have%
\begin{equation}
w_{nn}=\frac{f_{i}f_{j}}{\left\vert \nabla f\right\vert ^{2}}\left(
u_{ij}S^{-2}-4u_{i}S_{j}S^{-3}+6S_{i}S_{j}S^{-4}u-2S_{ij}S^{-3}u\right) .
\label{r1}
\end{equation}%
Now,%
\begin{equation*}
u_{ij}f_{i}f_{j}=\left\langle \nabla \left( u_{i}f_{i}\right) ,\nabla
f\right\rangle -f_{ij}u_{i}f_{j}.
\end{equation*}%
Note that by (\ref{m8}) and Proposition \ref{D2}, 
\begin{eqnarray}
\left\langle \nabla u,\nabla f\right\rangle  &=&-2u+\mathrm{Rm}\ast \mathrm{%
Rm}\ast \mathrm{Rm}  \label{u1} \\
&&+\nabla ^{2}\mathrm{Rm}\ast \mathrm{Rm}.  \notag
\end{eqnarray}%
Therefore,%
\begin{equation*}
\left\langle \nabla u,\nabla f\right\rangle =-2u+O\left( S^{3}\right) .
\end{equation*}%
Using (\ref{u1}) and (\ref{m8}) we similarly get 
\begin{eqnarray}
\left\langle \nabla \left\langle \nabla u,\nabla f\right\rangle ,\nabla
f\right\rangle  &=&-2\left\langle \nabla u,\nabla f\right\rangle +O\left(
S^{3}\right)   \label{u2} \\
&=&4u+O\left( S^{3}\right) .  \notag
\end{eqnarray}%
Finally, we have 
\begin{eqnarray*}
f_{ij}u_{i}f_{j} &=&\frac{1}{2}\left\langle \nabla u,\nabla f\right\rangle
-\left\langle \nabla u,\nabla S\right\rangle  \\
&=&-u+O\left( S^{3}\right) .
\end{eqnarray*}%
Hence, by (\ref{u1}) and (\ref{u2}) we conclude that 
\begin{equation}
u_{ij}f_{i}f_{j}=5u+O\left( S^{3}\right)   \label{u3}
\end{equation}%
and%
\begin{equation}
\frac{1}{\left\vert \nabla f\right\vert ^{2}}u_{ij}f_{i}f_{j}S^{-2}=\frac{5}{%
\left\vert \nabla f\right\vert ^{2}}w+O\left( Sf^{-1}\right) .  \label{r2}
\end{equation}%
Note that the second and the third term in (\ref{r1}) can be rewritten as 
\begin{eqnarray}
&&\frac{f_{i}f_{j}}{\left\vert \nabla f\right\vert ^{2}}\left(
-4u_{i}S_{j}S^{-3}+6uS_{i}S_{j}S^{-4}\right)   \label{r3} \\
&=&-\frac{4}{\left\vert \nabla f\right\vert ^{2}}\left( w_{i}f_{i}\right)
S_{j}f_{j}S^{-1}-\frac{2u}{\left\vert \nabla f\right\vert ^{2}}\left\langle
\nabla S,\nabla f\right\rangle ^{2}S^{-4}.  \notag
\end{eqnarray}%
The first term above is estimated by (\ref{Sc}) as%
\begin{equation*}
\frac{4}{\left\vert \nabla f\right\vert ^{2}}\left\vert \left\langle \nabla
w,\nabla f\right\rangle \left\langle \nabla S,\nabla f\right\rangle
\right\vert S^{-1}\leq cf^{-1}\left\vert \left\langle \nabla w,\nabla
f\right\rangle \right\vert .
\end{equation*}%
The second term can be computed as%
\begin{equation*}
\frac{2u}{\left\vert \nabla f\right\vert ^{2}}\left\langle \nabla S,\nabla
f\right\rangle ^{2}S^{-4}=\frac{2}{\left\vert \nabla f\right\vert ^{2}}%
w+O\left( Sf^{-1}\right) 
\end{equation*}%
by noting that 
\begin{equation*}
\left\langle \nabla S,\nabla f\right\rangle =-S+O(S^{2}),
\end{equation*}%
where we have used (\ref{m8}) and Proposition \ref{D2}. Thus, we have proved
that%
\begin{eqnarray}
&&\frac{f_{i}f_{j}}{\left\vert \nabla f\right\vert ^{2}}\left(
-4u_{i}S_{j}S^{-3}+6uS_{i}S_{j}S^{-4}\right)   \label{u4} \\
&\leq &c\,t_{0}^{-1}\,\left\vert \left\langle \nabla w,\nabla f\right\rangle
\right\vert -\frac{2}{\left\vert \nabla f\right\vert ^{2}}w+O\left(
Sf^{-1}\right) .  \notag
\end{eqnarray}

To estimate the last term in (\ref{r1}), we write%
\begin{eqnarray*}
S_{ij}f_{i}f_{j} &=&\left\langle \nabla \left\langle \nabla S,\nabla
f\right\rangle ,\nabla f\right\rangle -f_{ij}S_{i}f_{j} \\
&=&\left\langle \nabla \left\langle \nabla S,\nabla f\right\rangle ,\nabla
f\right\rangle -\frac{1}{2}\left\langle \nabla S,\nabla f\right\rangle +%
\frac{1}{2}\left\vert \nabla S\right\vert ^{2}.
\end{eqnarray*}%
Following a similar idea as in the proof of (\ref{u3}), one sees that $%
\left\langle \nabla S,\nabla f\right\rangle =-S+O(S^{2})$ and $\left\langle
\nabla \left\langle \nabla S,\nabla f\right\rangle ,\nabla f\right\rangle
=S+O\left( S^{2}\right).$ Therefore,%
\begin{equation*}
S_{ij}f_{i}f_{j}=\frac{3}{2}S+O\left( S^{2}\right).
\end{equation*}%
This implies that%
\begin{equation}
\frac{2}{\left\vert \nabla f\right\vert ^{2}}u\left( S_{ij}f_{i}f_{j}\right)
S^{-3}=\frac{3}{\left\vert \nabla f\right\vert ^{2}}w+O\left( Sf^{-1}\right).
\label{u5}
\end{equation}%
By (\ref{r1}), (\ref{r2}), (\ref{u4}) and (\ref{u5}) we obtain%
\begin{equation}
w_{nn}\leq c_{0}\,t_{0}^{-1}\,\left\vert \left\langle \nabla G,\nabla
f\right\rangle \right\vert +O(Sf^{-1}).  \label{r4}
\end{equation}

Note that by (\ref{Sc}),%
\begin{eqnarray*}
\left\langle \nabla G,\nabla \ln S\right\rangle  &=&\left\langle \nabla
G,\nabla \ln S\right\rangle _{\Sigma }+\frac{1}{\left\vert \nabla
f\right\vert ^{2}}\left\langle \nabla G,\nabla f\right\rangle \left\langle
\nabla \ln S,\nabla f\right\rangle  \\
&\leq &\left\langle \nabla G,\nabla \ln S\right\rangle _{\Sigma
}+c_{0}\,t_{0}^{-1}\,\left\vert \left\langle \nabla G,\nabla f\right\rangle
\right\vert 
\end{eqnarray*}%
and, using (\ref{2ff}),%
\begin{equation*}
H\,G_{n}\leq c_{0}\,t_{0}^{-1}\,\left\vert \left\langle \nabla G,\nabla
f\right\rangle \right\vert .
\end{equation*}%
Combining this with (\ref{c15}) and (\ref{r4}), we conclude that%
\begin{eqnarray}
\Delta _{\Sigma }G &\geq &\left\langle \nabla G,\nabla f\right\rangle
-c_{0}\,t_{0}^{-1}\,\left\vert \left\langle \nabla G,\nabla f\right\rangle
\right\vert   \notag \\
&&-2\left\langle \nabla G,\nabla \ln S\right\rangle _{\Sigma }+\theta
SG-c\,S\,f^{-1}.  \label{r4'}
\end{eqnarray}%
Here $\theta >0$ depends only on $n$, whereas the constants $c_{0}$ and $c$
depend only on $n$, $C_{0}$ from (\ref{CLY}) and $A_{0}$,....,$A_{K}$ from (%
\ref{Rm}), for some absolute constant $K$.

Now if the maximum of $G$ is achieved on $\Sigma \left( t_{0}\right) ,$ then
(\ref{c12}) and (\ref{G}) imply that $G\leq \frac{2\varepsilon }{3}.$
Otherwise, by the maximum principle we get $G\leq c\,t_{0}^{-1}$, for a
constant $c$ that is independent of $t_{0}$ and $T$.  Hence, by assuming $%
t_{0}$ to be large enough, one concludes that $G\leq \frac{2\varepsilon }{3}.
$ In either case, it shows that $G\leq \frac{2\varepsilon }{3}$ on $D\left(
T\right) \backslash D\left( t_{0}\right) .$ Now (\ref{G}) implies that%
\begin{equation*}
\sup_{\Sigma \left( T\right) }\left\vert \overset{\circ }{\mathrm{Rm}%
_{\Sigma }}\right\vert ^{2}S^{-2}<\varepsilon .
\end{equation*}%
This contradicts with (\ref{c14'}). So the assumption that $T<\infty $ is
false. Therefore, 
\begin{equation}
\sup_{\Sigma \left( t\right) }\left\vert \overset{\circ }{\mathrm{Rm}%
_{\Sigma }}\right\vert ^{2}S^{-2}<\varepsilon   \label{r5}
\end{equation}%
for all $t\geq t_{0}.$

We now claim that $S\geq C>0$ on $M.$ Indeed, if there exists $x\in
M\backslash D\left( t_{0}\right) $ with $S\left( x\right) <\frac{1}{2C_{1}},$
where $C_{1}$ is the constant in (\ref{S'}), then (\ref{S'}) implies that $%
S\,f\leq c$ along the integral curve of $\nabla f$ through $x.$ But this
contradicts with the fact that $\left( M,g,x_{k}\right) $ converges to $%
\mathbb{R}\times \mathbb{S}^{n-1}/\Gamma .$ In conclusion, $S$ is bounded
below by a positive constant. 

Let us assume that $z_k\rightarrow \infty$ is a sequence so that $\left(M,g,z_k\right)$ converges smoothly to 
$\mathbb{R}\times N$, where $\left(N,h\right)$ is a shrinking Ricci soliton. By (\ref{r5}) and the fact that $S\ge C>0$, it follows that  $\left(N,h\right) $ is isometric  to a quotient of the round sphere $\mathbb{S}^{n-1}$. 
By hypothesis  $\Sigma(t_k)$, the level set of $f$ containing $x_k$, converges to $\mathbb{S}^{n-1}\slash \Gamma$. Since all level sets $\Sigma(t)$ are diffeomorphic for $t\ge t_0$ large enough, we conclude that $\left(N,h\right)$ is isometric to $\mathbb{S}^{n-1}\slash \Gamma$, for any such sequence $z_k$. 

If $y_k\rightarrow \infty$ is an arbitrary sequence, according to Proposition 5.2 in \cite{Na}, there exist sequences $y_k^{+} $ and $y_k^{-}$ so that $\left(M,g,y_k^{\pm}\right)$ converge smoothly to shrinking Ricci solitons.  We have established that any such shrinking solitons are isometric to the same quotient of the round cylinder $\mathbb{R}\times\mathbb{S}^{n-1}\slash \Gamma$, so using Proposition 5.2 in \cite{Na} it follows that  $\left(M,g,y_k\right)$ converges itself to $\mathbb{R}\times \mathbb{S}^{n-1}\slash \Gamma$.

From here we get that 
\begin{equation}
\lim_{x\rightarrow \infty }\left\vert \overset{\circ }{\mathrm{Rm}_{\Sigma }}%
\right\vert \left( x\right) =0\;\;\;\text{and}\;\;\;\lim_{x\rightarrow
\infty }\left\vert S(x)-\frac{n-1}{2}\right\vert =0.  \label{r6}
\end{equation}%
We now strengthen the above conclusion and show that $M$ is asymptotic to $%
\mathbb{R}\times \mathbb{S}^{n-1}/\Gamma .$ For this, we first obtain an
explicit convergence rate for (\ref{r6}). According to (\ref{c15}), there
exists $\alpha >0$ depending only on $n$ so that 
\begin{equation}
\Delta _{F}G\geq \alpha G-cf^{-1}\text{ on }M\backslash D\left( t_{0}\right)
,  \label{r6'}
\end{equation}%
where $G:=\left\vert \mathrm{Rm}\right\vert ^{2}S^{-2}-\frac{2}{\left(
n-1\right) \left( n-2\right) }$ and $F:=f-2\ln S.$ We may assume that $%
\alpha \leq 1$. Define 
\begin{equation*}
H:=G-t_{0}^{\frac{\alpha }{2}}f^{-\frac{\alpha }{2}}.
\end{equation*}%
Then, choosing $t_{0}$ large enough, (\ref{r6}) implies that $H<0$ on $%
\Sigma \left( t_{0}\right) $ and $H\rightarrow 0$ at infinity. Furthermore,
it is easy to check that 
\begin{eqnarray*}
\Delta _{F}f^{-\frac{\alpha }{2}} &=&-\frac{\alpha }{2}\left( \Delta
_{F}f\right) f^{-\frac{\alpha }{2}-1}+\frac{\alpha }{2}\left( \frac{\alpha }{%
2}+1\right) \left\vert \nabla f\right\vert ^{2}f^{-\frac{\alpha }{2}-2} \\
&\leq &\frac{\alpha }{2}f^{-\frac{\alpha }{2}}+cf^{-\frac{\alpha }{2}-1} \\
&\leq &\frac{3}{4}\alpha f^{-\frac{\alpha }{2}}.
\end{eqnarray*}%
Hence, (\ref{r6'}) implies that%
\begin{eqnarray*}
\Delta _{F}H &\geq &\alpha H+\frac{\alpha }{4}t_{0}^{\frac{\alpha }{2}}f^{-%
\frac{\alpha }{2}}-cf^{-1} \\
&\geq &\alpha H
\end{eqnarray*}%
on $M\backslash D\left( t_{0}\right) $. For the last inequality, we used
that $\frac{\alpha }{2}<1$. Using the maximum principle, we now conclude
that $H\leq 0$ on $M\backslash D\left( t_{0}\right) $. Hence, this proves
that there exists $b_{0}$ depending only on $n$ so that 
\begin{equation}
\left\vert \overset{\circ }{\mathrm{Rm}_{\Sigma }}\right\vert \leq
cf^{-b_{0}}  \label{r7}
\end{equation}%
on $M.$

Define the tensor ${\mathrm{Q}}$ on $M$ by 
\begin{eqnarray*}
Q_{ijkl} &=&R_{ijkl}-\frac{S}{(n-1)(n-2)}\left(
g_{ik}g_{jl}-g_{il}g_{jk}\right) \\
&&+\frac{S}{\left( n-1\right) \left( n-2\right) }\left( g_{ik}\frac{%
f_{j}f_{l}}{\left\vert \nabla f\right\vert ^{2}}-g_{jk}\frac{f_{i}f_{l}}{%
\left\vert \nabla f\right\vert ^{2}}+g_{jl}\frac{f_{i}f_{k}}{\left\vert
\nabla f\right\vert ^{2}}-g_{il}\frac{f_{j}f_{k}}{\left\vert \nabla
f\right\vert ^{2}}\right) .
\end{eqnarray*}%
Observe that $Q_{ijkl}=R_{ijkl}$ if at least one of the indices $i,j,k,l$ is
equal to $n$, and 
\begin{equation*}
Q_{abcd}=\overset{\circ }{R}_{abcd}.
\end{equation*}%
By (\ref{r7}) and (\ref{a}), we obtain that

\begin{equation}
\left\vert {\mathrm{Q}}\right\vert \leq cf^{-b_{0}}  \label{r8}
\end{equation}%
on $M.$

We claim that for all $k\geq 1$ there exists $b_{k}>0,$ which depends only
on $n,$ and $c_{k}$ so that 
\begin{equation}
\left\vert \nabla ^{k}\mathrm{Rm}\right\vert \leq c_{k}f^{-b_{k}}.
\label{r10}
\end{equation}%
Indeed, for $x\in \Sigma \left( t\right) $ and $\theta :=t^{-\frac{1}{2}%
b_{0}}$ let $\phi $ be a cut-off on $B_{x}\left( \theta \right) $ so that $%
\phi =1$ on $B_{x}\left( \frac{\theta }{2}\right) $ and $\left\vert \nabla
\phi \right\vert \leq c\theta ^{-1}.$ Integrating by parts and using (\ref%
{m8}) and (\ref{Rm}), we have 
\begin{eqnarray*}
\int_{B_{x}\left( \theta \right) }\left\vert \nabla \mathrm{Q}\right\vert
^{2}\phi ^{2} &=&-\int_{B_{x}\left( \theta \right) }\left( \Delta
Q_{ijkl}\right) Q_{ijkl}\phi ^{2} \\
&&-\int_{B_{x}\left( \theta \right) }\left\langle \nabla Q_{ijkl},\nabla
\phi ^{2}\right\rangle Q_{ijkl} \\
&\leq &c\left( 1+\frac{1}{\theta }\right) \int_{B_{x}\left( \theta \right)
}\left\vert \mathrm{Q}\right\vert .
\end{eqnarray*}%
It follows from (\ref{r8}) that 
\begin{equation}
\int_{B_{x}\left( \theta \right) }\left\vert \nabla \mathrm{Q}\right\vert
^{2}\phi ^{2}\leq ct^{-\frac{1}{2}b_{0}}\mathrm{Vol}\left( B_{x}\left(
\theta \right) \right) .  \label{r11}
\end{equation}
As $\left( M,g\right) $ has bounded curvature, the volume comparison implies
that $\mathrm{Vol}\left( B_{x}\left( \theta \right) \right) \leq c\mathrm{Vol%
}\left( B_{x}\left( \frac{\theta }{2}\right) \right) .$ Together with (\ref%
{r11}), this proves that 
\begin{equation}
\inf_{B_{x}\left( \frac{\theta }{2}\right) }\left\vert \nabla \mathrm{Q}%
\right\vert ^{2}\leq ct^{-\frac{1}{2}b_{0}}.  \label{r12}
\end{equation}%
By (\ref{Rm}) we have that $\left\vert \nabla \left\vert \nabla \mathrm{Q}%
\right\vert ^{2}\right\vert \leq c.$ This together with (\ref{r12})
immediately leads to%
\begin{equation*}
\left\vert \nabla \mathrm{Q}\right\vert \leq cf^{-\frac{1}{4}b_{0}}.
\end{equation*}%
Therefore, as the hessian of $f$ is bounded, we get that 
\begin{eqnarray}
&&\nabla _{p}R_{ijkl}-\frac{\nabla _{p}S}{(n-1)(n-2)}\left(
g_{ik}g_{jl}-g_{il}g_{jk}\right)  \label{r122} \\
&&+\frac{\nabla _{p}S}{\left( n-1\right) \left( n-2\right) }\left( g_{ik}%
\frac{f_{j}f_{l}}{\left\vert \nabla f\right\vert ^{2}}-g_{jk}\frac{f_{i}f_{l}%
}{\left\vert \nabla f\right\vert ^{2}}+g_{jl}\frac{f_{i}f_{k}}{\left\vert
\nabla f\right\vert ^{2}}-g_{il}\frac{f_{j}f_{k}}{\left\vert \nabla
f\right\vert ^{2}}\right)  \notag \\
&=&O\left( f^{-\frac{1}{4}b_{0}}\right) .  \notag
\end{eqnarray}%
Tracing this formula, we obtain 
\begin{equation*}
\nabla _{p}R_{ik}-\frac{\nabla _{p}S}{n-1}\,g_{ik}-\frac{\nabla _{p}S}{n-1}%
\frac{f_{i}f_{k}}{\left\vert \nabla f\right\vert ^{2}}=O\left( f^{-\frac{1}{4%
}b_{0}}\right) .
\end{equation*}%
Tracing this in $p=k$, and using that 
\begin{equation*}
\left\vert \left\langle \nabla S,\nabla f\right\rangle \right\vert \leq c,
\end{equation*}%
we conclude from above that 
\begin{equation}
\left\vert \nabla S\right\vert \leq cf^{-\frac{1}{4}b_{0}}.  \label{r13}
\end{equation}%
By (\ref{r122}) and (\ref{r13}) it follows that 
\begin{equation}
\left\vert \nabla \mathrm{Rm}\right\vert \leq cf^{-\frac{1}{4}b_{0}}.
\label{r13'}
\end{equation}%
By induction on $k$ we get (\ref{r10}).

We now consider $\phi _{t}$ defined by 
\begin{eqnarray}
\frac{d\phi _{t}}{dt} &=&\frac{\nabla f}{\left\vert \nabla f\right\vert ^{2}}
\label{phi} \\
\phi _{t_{0}} &=&\mathrm{Id}\text{ \ on }\Sigma \left( t_{0}\right) .  \notag
\end{eqnarray}%
For a fixed $x\in \Sigma \left( t_{0}\right) $ we denote $S\left( t\right)
:=S\left( \phi _{t}\left( x\right) \right) ,$ where $t\geq t_{0}.$ Then%
\begin{eqnarray}
\frac{dS}{dt} &=&\frac{\left\langle \nabla S,\nabla f\right\rangle }{%
\left\vert \nabla f\right\vert ^{2}}  \label{rs} \\
&=&\frac{\Delta S-S+2\left\vert \mathrm{Ric}\right\vert ^{2}}{t-S}.  \notag
\end{eqnarray}%
Tracing (\ref{r7}) we get that 
\begin{equation*}
\left\vert R_{ab}-\frac{S}{n-1}g_{ab}\right\vert \leq cf^{-\delta },
\end{equation*}%
whereas by (\ref{r10}) we have 
\begin{equation*}
\left\vert \Delta S\right\vert \leq cf^{-\delta },
\end{equation*}%
for some $\delta >0$. Hence, (\ref{rs}) implies that 
\begin{equation*}
t\frac{dS}{dt}=\frac{2}{n-1}S^{2}-S+O\left( f^{-\delta }\right) .
\end{equation*}%
Consequently, the function $\rho :=S-\frac{n-1}{2}$ satisfies 
\begin{equation*}
t\rho ^{\prime }=\rho +\frac{2}{n-1}\rho ^{2}+O\left( f^{-\delta }\right)
\end{equation*}%
and $\rho \rightarrow 0$ at infinity. Integrating this in $t$ we find that
there exists $\delta >0$ so that $\left\vert \rho \left( t\right)
\right\vert \leq ct^{-\delta }.$ Hence, we have proved that 
\begin{equation}
\left\vert S-\frac{n-1}{2}\right\vert \leq cf^{-\delta }\text{ \ on }%
M\backslash D\left( t_{0}\right) .  \label{r14}
\end{equation}%
This and (\ref{r7}) imply that 
\begin{equation}
\left\vert R_{abcd}-\frac{1}{2\left( n-2\right) }\left(
g_{ac}g_{bd}-g_{ad}g_{bc}\right) \right\vert \leq cf^{-\delta },
\label{Rm_decay}
\end{equation}%
for some $\delta >0$ depending only on $n$.

We can now prove that $M$ is smoothly asymptotic to $\mathbb{R}\times 
\mathbb{S}^{n-1}/\Gamma .$ Indeed, for $\phi _{t}$ defined in (\ref{phi})
consider $\widetilde{g}\left( t\right) =\phi _{t}^{\ast }\left( g\right) $
on $\Sigma \left( t_{0}\right) ,$ the pullback of the metric $g$ on $\Sigma
\left( t\right) .$ Then%
\begin{eqnarray*}
\frac{d}{dt}\widetilde{g}_{ab}\left( t\right) &=&2\,\phi _{t}^{\ast }\left( 
\frac{f_{ab}}{\left\vert \nabla f\right\vert ^{2}}\right) \\
&=&\frac{1}{\left\vert \nabla f\right\vert ^{2}}\phi _{t}^{\ast }\left(
g_{ab}-2R_{ab}\right) .
\end{eqnarray*}%
Hence, by (\ref{Rm_decay}), 
\begin{equation*}
-ct^{-1-\delta }\widetilde{g}_{ab}\left( t\right) \leq \frac{d}{dt}%
\widetilde{g}_{ab}\left( t\right) \leq ct^{-1-\delta }\widetilde{g}%
_{ab}\left( t\right) .
\end{equation*}%
Integrating in $t$ implies that 
\begin{equation*}
\left\vert \widetilde{g}_{ab}\left( t\right) -g_{ab}^{\infty }\right\vert
\leq ct^{-\delta },
\end{equation*}%
where $g_{ab}^{\infty }$ is the round metric on $\Sigma \left( t_{0}\right)
. $ Note that (\ref{r13'}) implies decay estimates for $\left\vert \partial
^{k}\widetilde{g}_{ab}\right\vert $ for all $k.$ It is easy to see that this
implies $M$ is smoothly asymptotic to $\mathbb{R}\times \mathbb{S}%
^{n-1}/\Gamma .$ The theorem is proved.
\end{proof}

\section{\label{4-dim}Asymptotic geometry of four dimensional shrinkers}

We are now in position to prove Theorem \ref{dim4} in the introduction. For
the convenience of the reader, we restate it here.

\begin{theorem}
\label{Main_1} Let $\left( M,g,f\right) $ be a complete, four dimensional
gradient shrinking Ricci soliton with bounded scalar curvature $S.$ If $S$
is bounded from below by a positive constant on end $E$ of $M,$ then $E$ is
smoothly asymptotic to the round cylinder $\mathbb{R}\times \mathbb{S}%
^{3}/\Gamma,$ or for any sequence $x_{i}\in E$ going to infinity along an
integral curve of $\nabla f,$ $(M,g,x_{i})$ converges smoothly to $\mathbb{R}%
^{2}\times \mathbb{S}^{2}$ or its $\mathbb{Z}_2$ quotient. Moreover, the
limit is uniquely determined by the integral curve and is independent of the
sequence $x_i.$
\end{theorem}

\begin{proof}
Since $S$ is bounded, by Theorem \ref{MW}, $M$ has bounded curvature. Recall
that $\left( M,g\left( t\right) \right) $ is an ancient solution to the
Ricci flow defined on $\left( -\infty ,0\right) ,$ where 
\begin{equation*}
g\left( t\right) :=\left( -t\right) \phi _{t}^{\ast }g
\end{equation*}%
and $\phi _{t}$ is the family of diffeomorphisms defined by 
\begin{eqnarray*}
\frac{d\phi }{dt} &=&\frac{\nabla f}{\left( -t\right) } \\
\phi _{-1} &=&\mathrm{Id.}
\end{eqnarray*}%
For any sequence $\tau _{i}\rightarrow 0,$ consider the rescaled flow $%
\left( M,g_{i}\left( t\right) \right) $ for $t\,<0,$ where 
\begin{equation*}
g_{i}\left( t\right) :=\frac{1}{\tau _{i}}g\left( \tau _{i}t\right) .
\end{equation*}%
By Theorem 1.5 in \cite{Na}, for any $x_{0}\in \Sigma \left(
t_{0}\right)$, a subsequence of $\left( M,g_{i}\left( t\right) ,\left(
x_{0},-1\right) \right) $ converges smoothly to a gradient shrinking Ricci
soliton $\left( M_{\infty },g_{\infty }\left( t\right) ,\left( x_{\infty
},-1\right) \right) .$

Now for any sequence $x_{i}\in E$ going to infinity along an integral curve
of $\nabla f,$ obviously one may write $x_{i}:=\phi _{-\tau _{i}}\left(
x_{0}\right) $ for some point $x_{0}$ and $\tau _{i}\rightarrow 0.$ However,
as $g_{i}\left( -1\right) =\phi _{-\tau _{i}}^{\ast }g,$ we see that a
subsequence of $\left( M,g,x_{i}\right) $ converges to $\left( M_{\infty
},g_{\infty }\left( -1\right) ,x_{\infty }\right) .$ Since $x_{i}\rightarrow
\infty ,$ invoking Proposition 5.1 in \cite{Na} we conclude that $\left(
M_{\infty },g_{\infty }\left( -1\right) \right) $ splits as $\left( \mathbb{R%
},ds^{2}\right) \times \left( N,h\right) ,$ where $\left( N,h\right) $ is a
normalized three dimensional gradient shrinking Ricci soliton. Theorem \ref%
{dim3} implies that $\left( N,h\right) $ is isometric to a quotient of
either $\mathbb{S}^{3}$ or $\mathbb{R}\times \mathbb{S}^{2}.$ If the
quotient of $\mathbb{S}^{3}$ ever occurs, then Theorem \ref{cyl} implies
that $E$ is smoothly asymptotic to $\mathbb{R}\times \mathbb{S}^{3}/\Gamma .$
So we may assume that $N$ is never isometric to a quotient of $\mathbb{S}%
^{3}.$ In this case, for any sequence $x_{i}\in E$ going to infinity along
an integral curve of $\nabla f,$ a subsequence of $(M,g,x_{i})$ converges
smoothly to $\mathbb{R}\times N,$ where $N$ is either $\mathbb{R}\times 
\mathbb{S}^{2}$ or its $\mathbb{Z}_{2}$ quotient. However, by Remark 5.1 in 
\cite{Na}, such $N$ is uniquely determined by the integral curve. This
proves the theorem.
\end{proof}

We conclude with a rigidity result for four dimensional gradient shrinking K\"{a}hler Ricci soliton.

\begin{proposition}
\label{Kahler_nonneg}Let $\left( M,g,f\right) $ be a complete, non-flat,
four dimensional, gradient shrinking K\"{a}hler Ricci soliton with bounded
nonnegative Ricci curvature. Then $\left( M,g\right) $ is isometric to a
quotient of $\mathbb{R}^{2}\times \mathbb{S}^{2}.$
\end{proposition}

\begin{proof}
In view of (\ref{m8}), since $\left( M,g\right) $ has nonnegative Ricci
curvature, it follows that $S$ increases along each integral curve of $%
\nabla f.$ Hence, $S$ is bounded below by a positive constant. Since $M$
is K\"{a}hler, it can never be asymptotic to a quotient of the round cylinder.
In view of Theorem \ref{Main_1}, we conclude that
$\left( M,g\right)$ converges along each integral curve to 
$\left(\mathbb{R}^{2}\times \mathbb{S}^{2}\right)/\Gamma $ and $S$ must
converge to $1$ at infinity. In particular, this means that there exists a
compact set $K\subset M$ so that $S\leq 1$ on $M\backslash K,$ where the
compact set $K$ contains all critical points of $f.$

We diagonalize the Ricci curvature and denote the eigenvalues by $\alpha
\leq \beta.$ Then it follows that 
\begin{eqnarray*}
S^{2}-2\left\vert \mathrm{Ric}\right\vert ^{2} &=&4\left( \alpha +\beta
\right) ^{2}-4\left( \alpha ^{2}+\beta ^{2}\right) \\
&=&8\alpha \beta \geq 0.
\end{eqnarray*}

Hence, on $M\backslash K$ we have 
\begin{eqnarray}
\Delta _{f}S &=&S-S^{2}+S^{2}-2\left\vert \mathrm{Ric}\right\vert ^{2}
\label{p3} \\
&\geq &S-S^{2}  \notag \\
&\geq &0.  \notag
\end{eqnarray}

Without loss of generality, we may assume that $K=D\left( t_{0}\right) $ for
some $t_{0}>0.$ Then by the Stokes theorem, 
\begin{equation}
0\leq \int_{M\backslash D\left( t_{0}\right) }\left( \Delta _{f}S\right)
e^{-f}=-\int_{\Sigma \left( t_{0}\right) }\frac{\left\langle \nabla S,\nabla
f\right\rangle }{\left\vert \nabla f\right\vert }e^{-f}\leq 0,  \label{p4}
\end{equation}%
where the last inequality is because $\left\langle \nabla S,\nabla
f\right\rangle =2\mathrm{Ric}\left( \nabla f,\nabla f\right) \geq 0.$ It
follows from (\ref{p3}) and (\ref{p4}) that $S=1$ on $M\backslash K$ and the
eigenvalue $\alpha $ of the Ricci curvature is zero. By Corollary 1.3 of 
\cite{K}, $(M, g)$ is real analytic. Therefore, $S=1,$ $\alpha =0$ and $%
\beta =\frac{1}{2}$ on $M.$ The proposition follows from the de Rham
splitting theorem.
\end{proof}

\end{document}